\documentclass[12pt]{amsart}
\usepackage{amssymb,latexsym,eufrak,amsmath,amscd, graphicx}
\usepackage[dvipsnames,usenames]{color}
\usepackage[colorlinks=true,pagebackref,hyperindex]{hyperref}   
\hypersetup{ linkcolor=RawSienna, anchorcolor=BurntOrange,
citecolor=OliveGreen, filecolor=BlueViolet, menucolor=Yellow,
urlcolor=OliveGreen }

\usepackage{amsfonts}
\usepackage[all]{xypic}
\UseTips

\renewcommand{\to}[1][]{\xrightarrow{\ #1\ }}

\newcommand{\forget}[1]{}  

\makeatletter
\renewcommand{\theenumi}{\@roman\c@enumi}
\makeatother

\renewcommand{\phi}{\varphi}
\renewcommand{\epsilon}{\varepsilon}
\renewcommand{\theta}{\vartheta}

\def\ZZ{{\mathbf Z}}
\def\NN{{\mathbf N}}

\def\RR{{\mathbf R}}
\def\QQ{{\mathbf Q}}
\def\PP{{\mathbf P}}

\def\<{\langle}
\def\>{\rangle}

\def\bu{\mathbf{u}}

\def\cE{\mathcal{E}}
\def\cF{\mathcal{F}}
\def\cL{\mathcal{L}}
\def\cO{\mathcal{O}}

\def\cM{\mathcal{M}}

\DeclareMathOperator{\divisor}{div}
\DeclareMathOperator{\Gr}{Gr}
\def\Fl{F\ell}

\DeclareMathOperator{\Sym}{Sym}

 \DeclareMathOperator{\Pic}{Pic}

\newtheorem{lemma}{Lemma}[section]
\newtheorem{theorem}[lemma]{Theorem}
\newtheorem{corollary}[lemma]{Corollary}
\newtheorem{proposition}[lemma]{Proposition}
\newtheorem*{claim}{Claim}

\newtheorem*{KCT}{Klyachko's Classification Theorem}

\theoremstyle{definition}

\newtheorem{remark}[lemma]{Remark}

\newtheorem{example}[lemma]{Example}
\newtheorem{question}[lemma]{Question}

\theoremstyle{remark}
\newtheorem*{remark*}{Remark}
\newtheorem*{note*}{Note}

\begin{document}

\title{Positivity for toric vector bundles}

\author[M.~Hering]{Milena Hering}
\address{Institute of Mathematics and its Applications,
University of Minnesota, 400 Lind Hall, Minneapolis, MN 55455 USA \ \ \tt hering@ima.umn.edu}

\author[M. Musta\c{t}\u{a}]{Mircea~Musta\c{t}\u{a}}
\address{Department of Mathematics, University of Michigan,
East Hall, Ann Arbor, MI 48109, USA \ \ \tt mmustata@umich.edu}

\author[S.~Payne]{Sam Payne}
\address{Department of Mathematics, Stanford University, Bldg 380, Stanford, CA 94305, USA  \ \ \tt spayne@stanford.edu}

\markboth{M.~Hering, M.~Musta\c t\u a and S.~Payne}{Positivity
properties of toric vector bundles}

\begin{abstract}
We show that a torus-equivariant vector bundle on a complete toric variety is nef or ample if and only if its restriction to every invariant curve is nef or ample, respectively.  Furthermore, we show that nef toric vector bundles have a nonvanishing global section at every point and deduce that the underlying vector bundle is trivial if and only if its restriction to every invariant curve is trivial.  We apply our methods and results to study, in particular, the vector bundles $\cM_L$ that arise as the kernel of the evaluation map $H^0(X,L) \otimes \cO_X \rightarrow L$, for ample line bundles $L$. We give examples of 
twists of such bundles that are ample but not globally generated.

\bigskip

\noindent \textsc{R\'esum\'e}.  Nous prouvons qu'un fibr\'e vectoriel equivariant sur une vari\'et\'e torique compl\`ete est nef ou ample si et seulement si sa restriction \`a chaque courbe invariante est nef ou ample, respectivement.  Nous montrons \'egalement qu'\'etant donne un fibr\'e vectoriel torique nef $\cE$ et un point $x\in X$, il existe une section de $\cE$ non-nulle en $x$; on d\'eduit de cela que $\cE$ est trivial si et seulement si sa restriction \`a chaque courbe invariante est triviale.  Nous appliquons ces 
r\'esultats et m\'ethodes pour \'etudier en particulier les fibr\'es vectoriels $\cM_L$, 
 d\'efinis en tant que noyau des applications d'\'evaluation 
 $H^0(X,L) \otimes \cO_X \rightarrow L$, ou $L$
 est un fibr\'e en droites ample.  Finalement, nous donnons des exemples des fibr\'es vectoriels toriques qui sont amples mais non engendr\'{e}s par leur sections globales.
\end{abstract}

\thanks{2000\,\emph{Mathematics Subject Classification}.
 Primary 14M25; Secondary 14F05.
\newline The second author
  was partially supported by NSF grant
 DMS 0500127 and by a Packard Fellowship.
The third author was supported by the Clay Mathematics Institute}
\keywords{Toric variety, toric vector bundle}

\maketitle

\section{Introduction}

Let $X$ be a complete toric variety. If $L$ is a line bundle on $X$,
then various positivity properties of $L$ admit explicit
interpretations in terms of convex geometry. These interpretations can be
used to deduce special properties of toric line bundles. For
example, if $L$ is nef then it is globally generated. Moreover, $L$
is nef or ample if and only if the intersection number of $L$ with
every invariant curve is nonnegative or positive, respectively.  In this paper, we investigate the extent to which such techniques and results extend to equivariant vector bundles of higher rank.

Our first main result Theorem~\ref{char_nef} says that nefness and ampleness can be detected by
restricting to invariant curves also in higher rank. More precisely,
if $\cE$ is an equivariant vector bundle on $X$, then $\cE$ is nef or ample if and only if for every invariant curve $C$ on $X$, the
restriction $\cE\vert_C$ is nef or ample, respectively. Note that such a curve $C$
is isomorphic to $\PP^1$ (by convention, when considering \emph{invariant curves}, we assume that
they are irreducible), and therefore
$$\cE\vert_C\simeq\cO_{\PP^1}(a_1)\oplus\cdots\oplus\cO_{\PP^1}(a_r)$$
for suitable $a_1,\ldots,a_r\in\ZZ$. In this case $\cE\vert_C$ is
nef or ample if and only if all $a_i$ are nonnegative or positive, respectively. We apply the above result in Section~\ref{Qsection} to describe the Seshadri
constant of an equivariant vector bundle $\cE$ on a smooth toric
variety $X$ in terms of the decompositions of the restrictions of
$\cE$ to the invariant curves in $X$.

The characterization of nef and ample line bundles has an
application in the context of the bundles $\cM_L$ that appear as the kernel of the 
evaluation map $H^0(X,L) \otimes \cO_X \rightarrow L$, for globally generated line bundles $L$. We show that if $C$ is an invariant
curve on $X$, and $L$ is ample, then $\cM_L\vert_C$ is isomorphic to $\cO_{\PP^1}^{\oplus
a}\oplus\cO_{\PP^1}(-1)^{\oplus b}$ for nonnegative integers $a$ and
$b$. We then deduce that, for any ample line bundle $L'$ on $X$, the tensor product $\cM_L\otimes L'$ is nef.

Our second main result Theorem~\ref{sections} says that if $\cE$ is a nef equivariant
vector bundle on $X$ then, for every point $x\in X$, there is a
global section $s\in H^0(X,\cE)$ that does not vanish at $x$.  This
generalizes the well-known fact that nef line bundles on toric varieties are globally generated. 
On the other hand, we give examples of  ample toric vector bundles that are not globally generated
(see Examples~\ref{example12} and \ref{example13}).

The proof of Theorem~\ref{sections} relies on a description of toric vector bundles in terms of piecewise-linear families of filtrations, introduced by the third author in \cite{Payne}, that continuously interpolate the filtrations appearing in Klyachko's Classification Theorem \cite{Klyachko}.  As an
application of this result, we show that if $\cE$ is a toric vector bundle on a complete toric variety, then
$\cE$ is trivial (disregarding the equivariant structure) if and
only if its restriction to each invariant curve $C$ on $X$ is
trivial. This gives a positive answer to a question of V.~Shokurov.

In the  final section of the paper we discuss several open problems.

\subsection{Acknowledgments}
We thank Bill Fulton for many discussions,
Vyacheslav V.~Sho\-kurov for asking the question that led us to
Theorem~\ref{trivial}, and Jose Gonzalez for his comments on a preliminary version of this paper.

\section{Ample and nef toric vector bundles}

We work over an algebraically closed field $k$ of arbitrary
characteristic. Let $N\simeq\ZZ^n$ be a lattice, $M$ its dual lattice, $\Delta$ a fan in
$N_{\RR}=N\otimes_{\ZZ}\RR$, and $X=X(\Delta)$ the corresponding
toric variety. Then $X$ is a normal $n$-dimensional variety containing a dense open torus $T \simeq (k^*)^n$ such
that the natural action of $T$ on itself extends to an action of $T$
on $X$. In this section, we always assume that $X$ is complete, which means that the
support $|\Delta|$ is equal to $N_{\RR}$. For basic facts about
toric varieties we refer to \cite{Fulton}.

An \emph{equivariant} (or \emph{toric}) \emph{vector bundle} $\cE$
on $X$ is a locally free sheaf of finite rank on $X$ with a
$T$-action on the corresponding geometric vector bundle ${\mathbf
V}(\cE)={\mathcal Spec}({\rm Sym}(\cE))$ such that the projection
$\phi\colon {\mathbf V}(\cE)\to X$ is equivariant and $T$ acts linearly on the
fibers of $\phi$. In this case, note that the projectivized vector
bundle $\PP(\cE)={\mathcal Proj}({\rm Sym}(\cE))$ also has a
$T$-action such that the projection $\pi\colon \PP(\cE)\to X$ is
equivariant. Neither $\mathbf V(\cE)$ nor $\PP(\cE)$ is a toric
variety in general. However, every line bundle on $X$ admits an
equivariant structure, so if $\cE$
 splits as a sum of line bundles $\cE\simeq
\cL_1\oplus\cdots\oplus \cL_r$, then $\cE$ admits an equivariant
structure. In this case, both ${\mathbf V}(\cE)$ and $\PP(\cE)$
admit the structure of a toric variety; see \cite[pp. 58--59]{Oda}.

Note  that given an equivariant vector bundle $\cE$ on
$X$, we get an induced algebraic action of $T$ on the vector space of sections
$\Gamma(U_{\sigma},\cE)$, for every cone $\sigma\in\Delta$. In fact,
$\cE$ is determined as an equivariant vector bundle by the
$T$-vector spaces $\Gamma(U_{\sigma},\cE)$ (with the corresponding
gluing over $U_{\sigma_1}\cap U_{\sigma_2}$). Moreover, if $\sigma$
is a top-dimensional cone, and if $x_{\sigma}\in X$ is the
corresponding fixed point, then we get a $T$-action also on the
fiber $\cE\otimes k(x_{\sigma})$ of $\cE$ at $x_{\sigma}$ such that
the linear map $\Gamma(U_{\sigma},\cE)\to\cE\otimes k(x_{\sigma})$
is $T$-equivariant.

Given an algebraic action of $T$ on a vector space $V$, we get a
decomposition
$$V=\oplus_{u\in M}V_u,$$
where $V_u$ is the $\chi^u$-isotypical component of $V$, which means that $T$
acts on $V_u$ via the character $\chi^u$. For every $w\in M$, the
(trivial) line bundle $\cL_w:=\cO(\divisor\chi^{w})$ has a
canonical equivariant structure, induced by the inclusion of $\cL_w$
in the function field of $X$. For every cone $\sigma\in\Delta$ we
have
$$\Gamma(U_{\sigma},\cL_w)=\chi^{-w}\cdot k[\sigma^{\vee}\cap M],$$
and $\Gamma(U_{\sigma},\cL_w)_u=k\cdot\chi^{-u}$ (when
$w-u$ is in $\sigma^{\vee}$). Note that this is compatible with the 
convention that $T$ acts on $k\cdot\chi^u\subseteq
\Gamma(U_{\sigma},\cO_X)$ by $\chi^{-u}$ (we follow the standard convention
in invariant theory for the action of the group on the ring of functions; in toric geometry
one often reverses the sign of $u$ in this convention, making use of the fact that the torus is an abelian group). We also point out that if
$\sigma$ is a maximal cone, then $T$ acts on the fiber of $\cL_w$ at
$x_{\sigma}$ by $\chi^{w}$. It is known that every equivariant line
bundle on $U_{\sigma}$ is equivariantly isomorphic to some
$\cL_w\vert_{U_{\sigma}}$, where the class of $w$ in
$M/M\cap\sigma^{\perp}$ is uniquely determined.

For every cone $\sigma\in\Delta$, the restriction
$\cE\vert_{U_{\sigma}}$ decomposes as a direct sum of equivariant
line bundles $\cL_1\oplus\cdots\oplus\cL_r$. Moreover, each such
$\cL_i$ is equivariantly isomorphic to some
$\cL_{u_i}\vert_{U_{\sigma}}$, where the class of $u_i$ is uniquely
determined in $M/M\cap\sigma^{\perp}$. If $\sigma$ is a
top-dimensional cone, then in fact the multiset $\{u_1,\ldots,u_r\}$
is uniquely determined by $\cE$ and $\sigma$.

\bigskip

A vector bundle $\cE$ on $X$ is \emph{nef} or \emph{ample} if the line bundle $\cO(1)$ on $\PP(\cE)$ is nef or ample,
respectively.  For basic results about nef and ample vector bundles,
as well as the big vector bundles and $\QQ$-twisted vector bundles discussed below,
see \cite[Chapter~6]{positivity}.  It is well-known that a line bundle on a complete
toric variety is nef or ample if and only if its restriction to each invariant curve is so.
The following theorem extends this result to toric vector bundles.  Recall that every invariant
curve on a complete toric variety is isomorphic to $\PP^1$. Every vector bundle on
$\PP^1$ splits as a sum of line bundles
$\cO(a_1)\oplus\cdots \oplus\cO(a_r)$, for some integers
$a_1, \ldots, a_r$. Such a vector bundle is nef or ample if and only if all the $a_i$ are nonnegative or positive, respectively.

\begin{theorem}\label{char_nef}
A toric vector bundle on a complete toric variety is nef
or ample if and only if its restriction to every invariant curve is nef or ample,
respectively.
\end{theorem}

\begin{proof}
The restriction of a nef or ample vector bundle to a closed
subvariety is always nef or ample, respectively, so we must show the
converse.  First we consider the nef case. Suppose the restriction
of $\cE$ to every invariant curve is nef, so the degree of
$\cO_{\PP(\cE)}(1)$ is nonnegative on every curve in $\PP(\cE)$ that
lies in the preimage of an invariant curve in $X$. Let $C$ be an
arbitrary curve in $\PP(\cE)$. We must show that the degree of
$\cO_{\PP({\cE})}(1)$ on $C$ is nonnegative.  Let $v_1, \ldots, v_n$
be a basis for $N$, with $\gamma_i$ the one-parameter subgroup
corresponding to $v_i$.  Let $C_1$ be the flat limit of $t \cdot C$
as $t$ goes to zero in $\gamma_1$.  Hence $[C_1]$ is a one-dimensional
cycle in $\PP(\cE)$ that is linearly equivalent to $[C]$, and
$\pi(C_1)$ is invariant under $\gamma_1$.  Now let $C_i$ be the flat
limit of $t\cdot C_{i-1}$ as $t$ goes to zero in $\gamma_i$, for $2
\leq i \leq n$.  Then $[C_i]$ is linearly equivalent to $[C]$, and
$\pi(C_i)$ is invariant under the torus generated by $\gamma_1,
\ldots, \gamma_i$.  In particular, $[C_n]$ is linearly equivalent to
$[C]$ and every component of $C_n$ lies in the preimage of an
invariant curve in $X$.  Therefore the degree of $\cO_{\PP(\cE)}(1)$
on $C_n$, and hence on $C$, is nonnegative, as required.

Suppose now that the restriction of $\cE$ to every invariant curve is
ample. Note first that $X$ is projective. Indeed, the restriction of ${\rm det}(\cE)$
to every invariant curve on $X$ is ample, and since ${\rm det}(\cE)$ has rank one,
we deduce that ${\rm det}(\cE)$ is ample.

Let us fix  an ample line bundle $L$ 
on $X$, and choose an
integer $m$ that is greater than $(L \cdot C)$ for every invariant
curve $C$ in $X$.  The restriction of $\Sym^m(\cE) \otimes L^{-1}$
to each invariant curve is nef, and hence $\Sym^m(\cE) \otimes
L^{-1}$ is nef.  It follows that $\Sym^m(\cE)$ is ample, and hence
$\cE$ is ample as well \cite[Proposition~6.2.11 and
Theorem~6.1.15]{positivity}.
\end{proof}

\begin{remark}\label{general_variety}
Note that if $\cE$ is a vector bundle on an arbitrary complete variety $X$, then $\cE$ is nef
if and only if for every irreducible curve $C\subset X$, the restriction $\cE\vert_C$ is nef
(this simply follows from the fact that every curve in $\PP(\cE)$ is contained in some
$\PP(\cE\vert_C)$). The similar criterion for ampleness fails since there are non-ample line bundles 
that intersect positively every curve (see, for example,  \cite[Chap. I, \S 10]{Ample}).
However, suppose that $X$ is projective, and that we have finitely many curves $C_1,\ldots,C_r$
such that a vector bundle $\cE$ on $X$ is nef if and only if all $\cE\vert_{C_i}$ are nef. In this case,
arguing as in the above proof we see that a vector bundle $\cE$ on $X$ is ample if and only if 
all $\cE\vert_{C_i}$ are ample.
\end{remark}

\begin{remark}\label{non_toric}
The assumption in the theorem that $\cE$ is equivariant  is essential.
To see this, consider vector bundles $\cE$ on $\PP^n$ (see \cite[Section 2.2]{OSS}
for the basic facts that we use). If ${\rm rk}(\cE)=r$, then for every line $\ell$ in $\PP^n$
we have a decomposition 
$$\cE\vert_{\ell}\simeq\cO_{\PP^1}(a_1)\oplus\cdots\oplus\cO_{\PP^1}(a_r),$$
where we assume that the $a_i$ are ordered such that $a_1\geq\ldots\geq a_r$.
We put $\underline{a}_{\ell}=(a_1,\ldots,a_r)$. If we consider on $\ZZ^r$ the lexicographic order,
then the set $U$  of lines given by
$$U=\{\ell \in {\rm Gr}(1,\PP^n)\mid \underline{a}_{\ell}\leq \underline{a}_{\ell'}\,\text{for every}\, \ell'\in {\rm Gr}(1,\PP^n)\}$$
is open in the Grassmannian ${\rm Gr}(1,\PP^n)$. 
The vector bundle $\cE$ is \emph{uniform} if $U={\rm Gr}(1,\PP^n)$.

Suppose now that $\cE$ is a rank two vector bundle on $\PP^2$ that is not uniform (for an explicit example, see  \cite[Theorem 2.2.5]{OSS}). Let $(a_1,a_2)$ be the value of 
$\underline{a}_{\ell}$ for $\ell\in U$. If $\phi$ is a general element in ${\rm Aut}(\PP^n)$, then 
every torus-fixed line is mapped by $\phi$ to an element in $U$. It follows that if $\cE'=\phi^*(\cE)\otimes\cO_{\PP^2}(-a_2)$,
then $\cE'\vert_{\ell}$ is nef for every torus-invariant line $\ell$. On the other hand, if 
$\phi(\ell)\not\in U$, and if $\cE\vert_{\phi(\ell)}\simeq\cO_{\PP^1}(b_1)\oplus
\cO_{\PP^1}(b_2)$, then $b_2<a_2$ (note that $a_1+a_2=b_1+b_2=\deg(\cE)$), hence
$\cE'\vert_{\ell}$ is not nef.
\end{remark}

\begin{remark}\label{rem3}
Recall that a vector bundle $\cE$ is called \emph{big} if the line
bundle $\cO_{\PP(\cE)}(1)$ is big, which means that its Iitaka
dimension is equal to $\dim \PP(\cE)$.  The analogue of
Theorem~\ref{char_nef} does not hold for big vector bundles: there
are toric vector bundles $\cE$ such that the restriction of $\cE$ to
every invariant curve is big, but $\cE$ is not big.  Consider for
example $X=\PP^n$, for $n\geq 2$, and $\cE=T_{\PP^n}(-1)$.  An irreducible
torus-invariant curve in $\PP^n$ is a line. For such a line $\ell$
we have
$$\cE\vert_{\ell}\simeq\cO_{\ell}(1)\oplus\cO_{\ell}^{\oplus (n-1)}.$$
In particular, we see that $\cE\vert_{\ell}$ is big and nef.
However, $\cE$ is not big: the surjection $\cO_{\PP^n}^{\oplus (n+1)}\to T_{\PP^n}(-1)$
in the Euler exact sequence induces an 
embedding of  $\PP(\cE)$ in $\PP^n\times
\PP^n$, such that $\cO_{\PP(\cE)}(1)$ is the restriction of ${\rm
pr}_2^*(\cO_{\PP^n}(1))$. Therefore the Iitaka dimension of
$\cO_{\PP(\cE)}(1)$ is at most $n<\dim\,\PP(\cE)$.
\end{remark}

\begin{remark}\label{rem_nef_cone}
The argument in the proof of Theorem~\ref{char_nef} shows more
generally that a line bundle $L$ on $\PP(\cE)$ is nef if and only if
its restriction to every $\PP(\cE\vert_C)$ is nef, where $C$ is an
invariant curve on $X$. On the other hand, such a curve $C$ is
isomorphic to $\PP^1$, and $\cE\vert_C$ is completely decomposable.
Therefore $\PP(\cE\vert_C)$ has a structure of toric variety of
dimension ${\rm rk}(\cE)$. If we consider the invariant curves in
each such $\PP(\cE\vert_C)$, then we obtain finitely many curves
$R_1,\ldots,R_m$ in $\PP(\cE)$ (each of them isomorphic to $\PP^1$),
that span the Mori cone of $\PP(\cE)$. In particular, the
Mori cone of $\PP(\cE)$ is rational polyhedral.
\end{remark}

\section{$\QQ$-twisted bundles and Seshadri constants} \label{Qsection}

Recall that a $\QQ$-twisted vector bundle $\cE \< \delta \>$ on $X$
consists formally of a vector bundle $\cE$ on $X$ together with a
$\QQ$-line bundle $\delta \in \Pic(X) \otimes \QQ$.  Just as
$\QQ$-divisors simplify many ideas and arguments about positivity of
line bundles, $\QQ$-twisted vector bundles simplify many arguments
about positivity of vector bundles. We refer to
\cite[Section~6.2]{positivity} for details. One says that a
$\QQ$-twisted vector bundle $\cE \< \delta \>$ is nef or ample if
$\cO_{\PP(\cE)}(1) + \pi^* \delta$ is nef or ample, respectively. If
$Y$ is a subvariety of $X$, then the restriction of $\cE\< \delta
\>$ to $Y$ is defined formally as
\[
\cE\< \delta \> |_Y =  \cE|_Y \< \delta|_Y \>.
\]

\begin{remark} \label{Q version}
Since every $\QQ$-divisor is linearly equivalent to a $T$-invariant $\QQ$-divisor,
the proof of Theorem~\ref{char_nef} goes through essentially without change to show
that a $\QQ$-twisted toric vector bundle is nef or ample if and only if its restriction to every invariant curve is nef or ample, respectively.
\end{remark}

Suppose that $X$ is smooth and complete, $\cE$ is nef, and $x$ is a point in $X$.
Let $p\colon \widetilde X \rightarrow X$ be the blowup at $x$, with exceptional divisor $F$.
Recall that the \emph{Seshadri constant} $\epsilon(\cE, x)$ of $\cE$ at $x$ is defined to be the
supremum of the rational numbers $\lambda$ such that $p^*\cE \< -\lambda F \>$ is nef.
The global Seshadri constant $\epsilon(\cE)$ is defined as $\inf_{x\in X}\epsilon(\cE,x)$.
See \cite{Hacon} for background and further details about Seshadri constants of vector bundles.

We now apply Theorem~\ref{char_nef} to describe Seshadri constants
of nef toric vector bundles on smooth toric varieties. We start with
the following general definition. Suppose that $X$ is a complete toric variety, $\cE$ is a toric vector bundle on $X$, and $x\in X$
is a fixed point.  For each invariant curve $C$ passing through $x$, we have a decomposition
\[
\cE|_C \simeq \cO(a_1) \oplus \cdots \oplus \cO(a_r).
\]
We then define
$\tau(\cE,x):=\min\{a_i\}$, where the minimum ranges over all $a_i$,
and over all invariant curves passing through $x$. We also define
$\tau(\cE):=\min_x\tau(\cE,x)$, where the minimum is taken over all
fixed points of $X$. In other words, $\tau(\cE)$ is the minimum of the $a_i$,
where the minimum ranges over all invariant curves in $X$.  Note that Theorem~\ref{char_nef} says that
$\cE$ is nef or ample if and only if $\tau(\cE)$ is nonnegative or strictly positive, respectively.

We now give the following characterization of Seshadri constants of toric
vector bundles at fixed points, generalizing a result of Di Rocco for line bundles \cite{DR}.

\begin{proposition}\label{Seshadri}
Let $X$ be a smooth complete toric variety of dimension $n$, and $\cE$ a nef toric vector bundle on $X$.
If $x\in X$ is a torus-fixed point, then $\epsilon(\cE,x)$ is equal to $\tau(\cE,x)$.
\end{proposition}

\begin{proof}
Let $\lambda$ be a nonnegative rational number, and let $p\colon\widetilde X \rightarrow X$ be the blowup at a $T$-fixed point $x$, with exceptional divisor $F$.  Then $\widetilde X$ is a toric variety and $p$ is an equivariant morphism, so $p^* \cE \< -\lambda F \>$ is nef if and only if its restriction to every invariant curve is nef.

Let $\widetilde C$ be an invariant curve in $\widetilde X$.  If $\widetilde C$ is contained in $F$, then the restriction of $p^* \cE \< -\lambda F \>$ to 
$\widetilde{C}$ is isomorphic to a direct sum of copies of $\cO_{\PP^1}\< \lambda H\>$, where $H$ is the hyperplane class on $\PP^1$ (note that $\cO(-F)\vert_F$ is isomorphic to $\cO_{\PP^{n-1}}(1)$).

If $\widetilde C$ is not contained in the exceptional divisor, then $p$ maps $\widetilde C$ isomorphically onto an invariant curve $C$ in $X$.  If $C$ does not contain $x$ then the restriction of $p^* \cE \< -\lambda F \>$ to $\widetilde C$ is isomorphic to $\cE\vert_C$, which is nef.  On the other hand, if $x \in C$ then $(F \cdot \widetilde C) = 1$.  Then the restriction of $p^* \cE \< -\lambda F \>$ to $\widetilde C$ is isomorphic to $\cE\vert_C \< -\lambda H \>$.  Therefore, if the restriction of $\cE$ to $C$ is isomorphic to $\cO(a_1) \oplus \cdots\oplus\cO(a_r)$, then the restriction of $p^* \cE \< -\lambda F \>$ to
$\widetilde{C}$ is nef if and only if $\lambda \leq a_i$ for all $i$.  By Theorem~\ref{char_nef} for $\QQ$-twisted bundles (see Remark~\ref{Q version} above), it follows that $\epsilon(\cE,x) = \tau(\cE,x)$, as claimed.
\end{proof}

\begin{corollary}\label{global_Seshadri}
Under the assumptions in the proposition, the global Seshadri
constant $\epsilon(\cE)$ is equal to $\tau(\cE)$.
\end{corollary}

\begin{proof}
It is enough to show that the minimum of the local Seshadri constants $\epsilon (\cE, x)$ occurs at a fixed point $x \in X$.  Now, since $\cE$ is equivariant, $\epsilon(\cE, x)$ is constant on each $T$-orbit in $X$.  It then follows from the fact that the set of non-nef bundles in a family is parametrized by at most a countable union of closed subvarieties \cite[Proposition~1.4.14]{positivity} that if a torus orbit
 $O_\sigma$ is contained in the closure of an orbit $O_\tau$, then the local Seshadri constants of $\cE$ at points in $O_\sigma$ are less than or equal to those at points in $O_\tau$.  Therefore, the minimal local Seshadri constant must occur along a minimal orbit, which is a fixed point.
\end{proof}

For the following two corollaries, let $X$ be a smooth complete toric variety, with $\cE$ a toric vector bundle on $X$.

\begin{corollary}\label{cor1}
Let $p\colon \widetilde X \rightarrow X$ be the blowup of $X$ at a fixed point, with exceptional divisor $F$.  For every integer $m \geq 0$, we have $\tau(p^* \cE \otimes \cO(-mF))\geq
\min\{m,\tau(\cE) - m\}$.  In particular, if $\cE$ is ample then $p^* \cE \otimes \cO(-F)$ is nef.
\end{corollary}

\begin{proof}
Similar to the proof of Proposition~\ref{Seshadri}.
\end{proof}

\begin{corollary} \label{cor1'}
Let $q\colon X' \rightarrow X$ be 
the blowup of $X$ at $k$ distinct fixed points, with exceptional divisors $F_1, \ldots, F_k$.  If the Seshadri constant $\epsilon( \cE)$ is greater than or equal to two, then $q^* \cE \otimes \cO(-F_1 - \cdots - F_k)$ is nef.
\end{corollary}

\begin{proof}
Similar to the proof of Proposition~\ref{Seshadri}, using the fact that an invariant curve $C$ in $X$ contains at most two fixed points.
\end{proof}

As mentioned above, the set
of non-nef bundles in a family is at most a countable union of Zariski closed
subsets. However, in the toric
case we deduce from Theorem~\ref{char_nef} that the set of non-nef bundles is closed.

\begin{corollary}\label{cor11}
Let $S$ be a variety, and let $\cE$ be a vector bundle on $X \times S$ such that, for every point $s \in S$, the restriction $\cE_s$ of $\cE$ to $X \times \{s \}$ admits the structure of a toric vector bundle.  Then the set of points $s$ such that $\cE_s$ is  nef is open in $S$.
\end{corollary}

\begin{proof}
By Theorem~\ref{char_nef}, it is enough to prove the
corollary when $X=\PP^1$. In this case the assertion is clear, since
a vector bundle $\cF \simeq \cO(a_1) \oplus \cdots \oplus \cO(a_r)$
on $\PP^1$ is nef if and only if $h^0(\PP^1,\cF^\vee(-1))=0$.
\end{proof}

\section{The bundles $\cM_L$} \label{M_L}

Vector bundles appear naturally in the study of
 linear series on a projective variety. For example, suppose that $L$ is a globally
generated line bundle on $X$. The kernel $\cM_L$ of the
evaluation map $H^0(X,L)\otimes\cO_X\to L$ is a vector bundle 
whose behavior is closely 
related to the geometry of $L$.
 If $L$ is very ample,
then the projective normality of $X$ in the embedding given by $L$,
or the minimal degree in the syzygies
 of the ideal of $X$ are governed by properties of $\cM_L$
(see \cite{Gre1} and \cite{Gre2}). Note that if $X$ is a toric variety, 
and if we fix an equivariant structure on $L$,
then  $\cM_{L}$  becomes an equivariant
vector bundle. 

Recall the following well-known question about linear series on
smooth toric varieties. It would be quite interesting to understand what conditions would guarantee a
positive answer. This question motivates our study of the
vector bundles $\cM_L$.

\begin{question}\label{q1}${\rm (}$\cite{Oda1}${\rm )}$
If $L_1$ and $L_2$ are ample line bundles on a smooth projective
toric variety $X$, is the multiplication map
\begin{equation}\label{mult}
H^0(X,L_1)\otimes H^0(X,L_2)\to H^0(X,L_1\otimes L_2)
\end{equation}
 surjective ?
\end{question}

Fakhruddin \cite{Fakhruddin} proved that this question has a positive 
answer for an ample line bundle $L_1$ and a globally 
generated line bundle $L_2$  on a smooth toric surface. Recently,
Haase, Nill, Pfaffenholz and Santos \cite{HNPS}
 were able to prove this 
for arbitrary toric surfaces. Moreover it is well known 
\cite{EwaldWessels, LiuTrotterZiegler, BrunsGubeladzeTrung} 
that for an ample line bundle $L$ on a 
possibly singular toric variety of dimension $n$,
the multiplication map 
\begin{equation}
H^0(X,L^m)\otimes H^0(X,L)\to H^0(X,L^{m+1})
\end{equation}
is surjective for $m\geq n-1$.

\begin{remark}\label{reformulation_Minkowski}
Question~\ref{q1} can be restated in terms of Minkowski sums of polytopes, as follows.
Recall that an ample divisor $D$ on a toric variety corresponds 
to a lattice polytope $P\subset M_{\RR}$.  
Let $P_1$, $P_2$ be the lattice polytopes in $M_{\RR}$ 
corresponding to ample divisors $D_1$ and $D_2$ on $X$. Question \ref{q1}
is equivalent to the question whether the natural addition map
\begin{equation}\label{eq:polytopes}
(P_1\cap M) \times (P_2\cap M) \to (P_1+P_2)\cap M
\end{equation}
 is surjective. 
\end{remark}

\begin{remark}\label{rem4}
Question~\ref{q1} has a positive answer in general if and only if it has a
positive answer whenever $L_1=L_2$. Indeed, given two line bundles
$L_1$ and $L_2$ on the smooth toric variety $X$ we may consider the
toric variety $Y=\PP(L_1\oplus L_2)$ (note that $Y$ is smooth, too).
Since $\pi_*\cO_Y(m)={\rm Sym}^m(L_1\oplus L_2)$, it follows that if
$$H^0(Y,\cO_Y(1))\otimes H^0(Y,\cO_Y(1))\to H^0(Y,\cO_Y(2))$$
is surjective, then (\ref{mult}) is surjective, too.
\end{remark}

\begin{remark}
The argument in the previous remark can be restated combinatorially, as follows.
Let $D_1$ and $D_2$ be ample $T$-Cartier divisors, with 
$P_1$ and $P_2$ the corresponding lattice polytopes. 
Let $Q \subset M_{\RR} \times \RR$ be the Cayley sum of $P_1$ and $P_2$, which is the
convex hull of $(P_1 \times \{ 0 \}) \cup (P_2 \times \{1\})$. 
If the addition map $(Q\cap M) \times (Q\cap M) \to 2Q\cap M$ is 
surjective, then so is the map (\ref{eq:polytopes}) above. Note that 
$Q$ is the polytope corresponding to the line bundle $\cO_Y(1)$ as in 
the previous remark, and if $X$ is smooth then the toric variety corresponding to 
$Q$ is smooth.  
\end{remark}  

\bigskip

We now turn to the study of the vector bundles $\cM_L$.
Let $X$ be a complete toric variety, and $L$ a globally generated
line bundle on $X$. Let $\cM_L$ be the kernel of the evaluation map
associated to $L$:
$$\cM_L:=\ker\left(H^0(X,L)\otimes\cO_X\to L\right).$$
Since $L$ is globally generated, $\cM_L$ is a vector bundle of rank
$h^0(L)-1$. By definition, we have an exact sequence
\begin{equation}\label{sequence_ML}
0\to \cM_L\to H^0(X,L)\otimes\cO_X\to L\to 0.
\end{equation}

If $L'$ is another globally generated line bundle, then $H^i(X,L')$ vanishes for $i$ greater than zero.  Therefore, by tensoring with
$L'$ the exact sequence (\ref{sequence_ML}),
we see that the multiplication map $H^0(X,L)\otimes H^0(X,L')\to
H^0(X,L\otimes L')$ is surjective if and only if $H^1(X,\cM_L\otimes
L')=0$. We also  get from this exact sequence that in general
$H^i(X,\cM_L\otimes L')=0$ for every $i\geq 2$.

\begin{remark}
Suppose that $P_1$ and $P_2$ are the lattice polytopes in $M_{\RR}$ 
corresponding to the ample divisors $D_1$ and $D_2$ on $X$, as in
Remark~\ref{reformulation_Minkowski}. If $L_1=\cO(D_1)$ and $L_2=\cO(D_2)$, then
the points $w$ in $(P_1+P_2)\cap M$ that are not sums of lattice points 
of $P_1$ and $P_2$ correspond exactly to the torus weights $w$ for
which $H^1(X,\cM_{L_1}\otimes L_2)_w \neq 0$.
\end{remark}

\begin{proposition}\label{prop1}
If $L$ and $L'$ are line bundles on $X$, with $L$ ample, then
 for every fixed point $x \in X$ we have
\[
\tau(\cM_L \otimes L', x) = \tau(L',x) - 1.
\]
In particular,  $\cM_L \otimes L'$ is nef if and only if $L'$ is ample, and 
 $\cM_L \otimes L'$ is ample if and only if
$(L' \cdot C) \geq 2$ for every invariant curve $C$.
\end{proposition}

We need to describe the restriction of $\cM_L$ to the invariant curves on $X$. We do this recursively,
by first restricting $\cM_L$ to the prime invariant divisors in $X$.

\begin{proposition}\label{prop2}
Let $L$ be an ample line bundle on the projective toric variety $X$.
If $D$ is a prime invariant divisor in $X$, then
$$(\cM_L)\vert_D\simeq \cM_{L\vert_D}\oplus \cO_D^{\oplus m},$$ where
$m=h^0(X,L\otimes\cO(-D))$.
\end{proposition}

\begin{proof}
 We have an exact sequence
$$0\to\cO(-D)\to\cO_X\to\cO_D\to 0.$$
After tensoring with $L$, and using the fact that
$H^1(X,L\otimes\cO(-D))=0$ (see, for example, Corollary 2.5 in
\cite{Mus}), we deduce that the restriction map $H^0(X,L)\to
H^0(D,L\vert_D)$ is surjective. This restriction map induces a
commutative diagram with exact rows

\[
\begin{CD}
0@>>> (\cM_L)\vert_D@>>>H^0(X,L)\otimes\cO_D@>>>L_D@>>>0\\
@. @VV{\phi}V @VV{\psi}V@| \\
0 @>>>\cM_{L\vert_D}@>>>H^0(D,L\vert_D)\otimes\cO_D@>>>L_D@>>> 0.
\end{CD}
\]

\noindent We have seen that $\psi$ is surjective (hence also split), and
therefore $\phi$ is also surjective and split, and
$\ker(\phi)\simeq\ker(\psi)$. The proposition follows.  \end{proof}

\begin{corollary}\label{cor2}
If $C$ is an invariant curve on $X$, then
\begin{equation}\label{eq2}
(\cM_L)\vert_C\simeq \cO_{\PP^1}^{\oplus
a}\oplus\cO_{\PP^1}(-1)^{\oplus b},
\end{equation}
 where $a=h^0(L)-(L\cdot C)-1$
and $b=(L\cdot C)$.
\end{corollary}

\begin{proof}
 We can find a sequence of irreducible torus-invariant subvarieties
\[
C=V_1\subset V_2\subset\cdots\subset V_n=X,
\]
with $\dim(V_i)=i$ for all $i$. Applying the proposition $(n-1)$ times, we see that
$(\cM_L)\vert_C\simeq \cM_{L\vert_C}\oplus \cO_C^{\oplus a}$ for
some $a$.

On the other hand, on $C\simeq \PP^1$ we have
$\cM_{\cO(m)}\simeq\cO(-1)^{\oplus m}$ for every $m>0$ (note that
$\cM_{\cO(m)}$ is a vector bundle of rank $m$ and degree $-m$ such
that $h^0(\cM_{\cO(m)})=0$). Hence $(\cM_L)\vert_C$ has an
expression as in (\ref{eq2}), and $a$ and $b$ can be determined using the fact that the rank of $\cM_L$ is $h^0(L)-1$ and $\deg((\cM_L)\vert_C)=-(L\cdot C)$.
\end{proof}

\begin{proof}[Proof of Proposition~\ref{prop1}]
The first assertion
in the proposition follows directly from
Corollary~\ref{cor2}. Note that for every 
invariant curve $C$, the restriction $\cM_L\vert_C$ has $\cO(-1)$ as
a factor (this is a consequence of the fact that $(L\cdot C)>0$).
The last two assertions in the proposition follow from the first one using
Theorem~\ref{char_nef}.
\end{proof}

\bigskip

In light of Proposition~\ref{prop1}, and motivated by
Question~\ref{q1}, we see that it would be desirable to get
conditions on nef toric vector bundles that would guarantee the
vanishing of higher cohomology. It is well-known that the higher
cohomology of a nef line bundle on a toric variety vanishes.
However, as the following example shows,
 this fails in higher rank even for ample toric vector bundles.
 
 \begin{example}  \label{non_vanishing}
Let $P= \< (0,0,0,0), (1,0,0,0), (0,1,0,0), (0,0,1,0), 
(1,1,1,3) \>$ and let $L$ be the corresponding ample line bundle
on the toric variety $X$.  By Proposition~\ref{prop1}, the toric vector bundle $\cM_L \otimes L^2$ is ample on $X$.  The product map $H^0(X,L) \otimes H^0(X, L^2) \rightarrow H^0(X, L^3)$ is not surjective, since the point $(1,1,1,1) \in 3P$ cannot be written as the sum of a lattice point in $P$ and one in $2P$.  Therefore $H^1(X, \cM_L \otimes L^2)$ is nonzero.
 \end{example}

\noindent The toric variety $X$ in the example above appears also in Example~\ref{example12}, where we show that the ample vector bundle $\cM_{L^3} \otimes L^2$ is not globally generated.  This toric variety is singular, but there are also ample vector bundles with nonvanishing higher cohomology on smooth toric varieties, including projective spaces.  We thank O. Fujino for pointing out the following example. 

 \begin{example}\label{Fujino}
Let $F: \PP^n \rightarrow \PP^n$ be the toric morphism induced by multiplication by $q$ on $N_\RR$, where $q$ is an integer greater than $n+1$.  This map $F$ is known as the $q^{\rm th}$ toric Frobenius morphism; see Remark~\ref{rem:Frob} below.  Let $T$ be the tangent bundle on $\PP^n$, and $K \simeq \cO(-n-1)$ the canonical bundle.  The restriction of the pullback $F^* T$ to each invariant curve splits as  $\cO(q)^{\oplus (n-1)}\oplus\cO(2q)$.  We deduce that for $1 \leq j \leq n$, $\wedge^j F^*T \otimes K$ is ample.  Now, by \cite[Proposition~3.5]{Fujino}, if
$H^i(\PP^n, \wedge^j F^* T \otimes K)=0$ then 
$H^i(\PP^n, \wedge^j T \otimes K)=0$. On the other hand, $H^i(\PP^n, \wedge^j T \otimes K)$ is nonzero exactly when $i = n-j$, by Serre duality.  In particular, $\wedge^j F^* T \otimes K$ is an ample vector bundle on $\PP^n$ with nonvanishing higher cohomology for $1 \leq j < n$.
 \end{example}

We mention that L.~Manivel proved in \cite{Man} some interesting
vanishing results for ample (but not necessarily toric) vector
bundles on smooth toric varieties. He showed that if $\cE$ is such a
vector bundle, then $H^i(X,\Omega_X^p\otimes S^j(\cE))=0$ for every
$i\geq {\rm rk}(\cE)$, and $H^i(X,\Omega_X^p\otimes\wedge^j(\cE))=0$
for every $i>{\rm rk}(\cE)-j$. Further vanishing theorems for not necessarily
equivariant vector bundles and reflexive sheaves on toric varieties
have been proved by Fujino \cite{Fujino}.  However, these results do
not apply in our setting, to give vanishing for $H^1(X, \cM_L \otimes L')$ when $L$ and $L'$ 
are ample. An
interesting question is whether one could get stronger vanishing theorems for equivariant vector bundles.

\bigskip

We now turn to the question of the global generation of the bundles of the form $\cM_L\otimes
L'$. Let $L$ and $L'$ be globally generated line bundles on the complete toric variety $X$. 
We have seen that 
if both $L$ and $L'$ are ample then
$\cM_L \otimes L'$ is nef.  Furthermore, if $(L' \cdot C)\geq 2$ for every $T$-invariant curve $C$ then $\cM_L \otimes L'$ is ample.  We now give a necessary and sufficient combinatorial condition for
 $\cM_L \otimes L'$ to be generated by global sections.  As an application of this combinatorial criterion, we give examples of nef and ample toric vector bundles that are not globally generated.

Fix $T$-Cartier divisors $D$ and $D'$ on $X$ such that $L\simeq\cO(D)$ and $L'\simeq\cO(D')$.
This makes $L$ and $L'$ equivariant line bundles, and 
therefore $\cM_L$ and $\cM_L \otimes L'$ become toric vector bundles with the induced equivariant structures.  It is convenient to work with the negatives of the usual lattice polytopes associated 
to $D$ and $D'$, so we put
$P=-P_{D}$ and $P'=-P_{D'}$ .
Therefore $P$ is the convex hull of the lattice points $u \in M$ such that $H^0(X,L)_u$ is not zero, and similarly for $P'$ (recall that we assume that $L$ and $L'$ are globally generated).  For each maximal cone $\sigma \in \Delta$, let $u_\sigma$ and $u'_\sigma$ be the lattice points such that the restrictions of $D$ and $D'$ to $U_\sigma$ are equal to ${\rm div}(\chi^{u_\sigma})$ and ${\rm div}
(\chi^{u'_\sigma})$, respectively.
Thus $P$ is the convex hull of the $u_{\sigma}$, and similarly for $P'$.
Note that $H^0(X,L)_u=k\cdot \chi^{-u}$ if $u\in P\cap M$, and it is zero otherwise.

\begin{proposition}  \label{prop:globgen}
If $L$ and $L'$ are globally generated, then the vector bundle $\cM_L \otimes L'$ is generated by global sections if and only if $P \cap (u + u'_\sigma - P')$ contains at least two lattice points for every lattice point $u \in P \smallsetminus \{u_\sigma\}$ and every maximal cone $\sigma \in \Delta$.
\end{proposition}

\noindent Note that both $P$ and $(u + u'_\sigma - P')$ contain $u$, so the condition in Proposition~\ref{prop:globgen} is that the intersection of $P$ with $(u+u'_{\sigma}-P')$ should contain at least one other lattice point.

\begin{proof}
Tensoring the exact sequence (\ref{sequence_ML}) with $L'$ 
gives an exact sequence
$$0\to H^0(X, \cM_L\otimes L')_w\to (H^0(X, L)\otimes H^0(X, L'))_w\to H^0(X, L\otimes L')_w,$$
for every $w \in M$.
It follows that
$H^0(X, \cM_L\otimes L')_w$
consists of the sums  $\sum_u a_u\chi^{-u}\otimes\chi^{u-w}$ with $\sum_ua_u=0$, where the sum is over those  $u\in (P\cap M)$ with $w-u\in P'$. 

For every maximal cone $\sigma$ in $\Delta$, let
 $x_\sigma$ be the unique $T$-fixed point in $U_\sigma$. The vector bundle 
 $\cM_L\otimes L'$ is globally generated if and only if for every such $\sigma$, the image of
 $H^0(X, \cM_L\otimes L')$ in the fiber at $x_{\sigma}$ has dimension ${\rm rk}(\cM_L\otimes L')=h^0(L)-1$. Let us fix $w\in M$.
 Since $D'={\rm div}(\chi^{u'_{\sigma}})$ on $U_{\sigma}$, 
 it follows that the image of the section
 $s=\sum_ua_u\chi^{-u}\otimes\chi^{u-w}\in H^0(X, \cM_L\otimes L')_w$ in the fiber at $x_{\sigma}$
of $H^0(X, L)\otimes L'$ is $a_{w-u'_{\sigma}}\chi^{u'_{\sigma}-w}$ if $w-u'_{\sigma}$ is in $P$.

Suppose now that $w-u'_{\sigma}\in P$.
If there is at most one $u\in P\cap (w-P')\cap M$, then in the sum defining $s$ we have at most one term, and $a_{w-u'_{\sigma}}=0$ for every section
$s$ as above. Otherwise, $\chi^{u'_{\sigma}-w}$ lies in the image of $H^0(\cM_L\otimes L')_w$.
Note that if $u\in P$ and $u'\in P'$ are such that $u+u'=u_{\sigma}+u'_{\sigma}$, then 
$u=u_{\sigma}$ and $u'=u'_{\sigma}$ (this follows since $u_{\sigma}$ and $u'_{\sigma}$
are vertices of $P$ and $P'$, respectively). This shows that if $w - u'_{\sigma}=u_{\sigma}$, then 
$\#(P\cap (w-P')\cap M)=1$. It follows from the above discussion that the image of
$H^0(X, \cM_L\otimes L')$ in the fiber at $x_{\sigma}$ has rank $h^0(L)-1$ if and only if
for every lattice point $u$ in $P\smallsetminus\{u_{\sigma}\}$ we have
$\#(P\cap (u'_{\sigma}+u-P')\cap M)\geq 2$.
\end{proof}

\begin{corollary} Let $L$ be a globally generated line bundle on a toric 
variety $X$. Then $\cM_{L}\otimes L$ is globally generated. 
\end{corollary}
 
\begin{proof} For every maximal cone $\sigma$, and for all $u\in P\cap M 
\smallsetminus  \{u_{\sigma}\}$, the lattice points $u$ and $u_{\sigma}$ 
are contained in $P\cap 
\left(u+u_{\sigma}-P\right)$, and the assertion follows from 
 Proposition \ref{prop:globgen}.
\end{proof}

\begin{remark}\label{rem5}
The above corollary can also be deduced using the
Koszul complex associated to the evaluation map of $L$, which shows that $\cM_L$
is a quotient of $\left(\wedge^2H^0(X,L)\right)\otimes L^{-1}$.
\end{remark}

\begin{remark}\label{rem:Frob} For every positive integer $q$, let $F_q\colon X\to X$
be the $q^{\rm th}$ \emph{toric Frobenius morphism}: this is induced by 
 multiplication by $q$ on $N_{\RR}$. 
The name is due to the fact that if $k$ has characteristic $p$ then $F_p$ is the relative Frobenius morphism on $X$. If $W\subseteq H^0(X,L^q)$ is the vector subspace generated by $s^{\otimes q}$, for
$s\in H^0(X,L)$, then we have an exact sequence of equivariant vector bundles
$$0\to F_q^*(\cM_L)\to W\otimes\cO_X\to L^q\to 0.$$
 Arguing as in the proof of Proposition~\ref{prop:globgen},
 one can show that $F_q^*(\cM_L)\otimes L'$ is globally generated if 
 and only if, for every maximal cone $\sigma$ and 
 every lattice point $u\in P\smallsetminus\{u_{\sigma}\}$, the
 set $P\cap\frac{1}{q}(u'_{\sigma}+qu-P')$ contains at least two lattice points.
\end{remark}

Our main application of Proposition \ref{prop:globgen} is to give
examples of toric vector bundles that are ample or nef, but not globally 
generated.

\begin{example}\label{example11}
 Let $P = \< (0,0,0), (1,0,0), (0,1,0), (1,1,2)\>$, and let 
 $L$, $X$ be the ample line bundle and toric variety associated to $-P$.
 \footnote{Note that the polarized toric varieties associated to $P$ and to $-P$
 are canonically isomorphic.} 
 Then $\cM_{L^2}\otimes L$ is nef by Proposition 
\ref{prop1}, but not globally generated, since 
$u = (1,1,1) \in 2P$, and the maximal cone corresponding to $(0,0,0)$ 
violates the condition in Proposition \ref{prop:globgen}.
If  $f\colon Y\to X$ is a resolution of singularities, then 
$f^*(\cM_{L^2}\otimes L)$ gives an example of a nef 
but not globally generated toric bundle on a smooth toric threefold.
\end{example}

By going to dimension four, we can similarly get an example of an ample
toric vector bundle that is not globally generated (note, however, that in this case the
toric variety is not smooth).

\begin{example}\label{example12}
 Let $P= \< (0,0,0,0), (1,0,0,0), (0,1,0,0), (0,0,1,0), 
(1,1,1,3) \>$,  and let $L$, $X$ be the  ample line bundle
and the toric variety associated to $-P$. Note that $\cM_{L^3}\otimes L^2$ is ample by 
Proposition \ref{prop1}. However, it is not globally generated, since 
 \[ 3P\cap \left(\left(1,1,1,1\right) -2P\right) \cap M 
= \left\{ \left( 1,1,1,1\right)\right\},\]
hence the condition in Proposition~\ref{prop:globgen} is not satisfied for 
 the 
maximal cone $\sigma$ corresponding to $(0,0,0,0)$, and for 
$u=(1,1,1,1)\in 3P$.
\end{example}

We can get similar examples in dimension three, if we consider
also bundles of the form $F_q^*(\cM_L)\otimes L'$.

\begin{example}\label{example13}
 Let $P=\<(0,0,0), (1,0,0), (0,1,0), (1,1,3)\>$.
If $F_2\colon X\to X$ is the toric Frobenius morphism as in Remark \ref{rem:Frob}, 
then $\cE:=\left(F_2^*\cM_{L^2}\right)\otimes L^3$ is ample but not 
globally generated. Indeed, for every $q$ and every invariant curve $C$ on $X$
we have  $F^*_q\cM_{L^2}\vert_C\simeq
\cO_{\PP^1}^{\oplus
a}\oplus\cO_{\PP^1}(-q)^{\oplus b}$, with $a,b$ given by Corollary
\ref{cor2}. Hence $\cE$ is ample by Theorem \ref{char_nef}. 
On the other hand, 
 if $\sigma$ is the maximal 
cone corresponding to $(0,0,0)$,
since
$$2P\cap \left((1,1,1)-\frac{3}{2}P\right)\cap M=\{(1,1,1)\},$$
we see by Remark~\ref{rem:Frob} that
$\cE$ is not globally generated. 
\end{example}

The following proposition shows that when $L$ is normally generated, 
the vector bundles appearing in the preceding examples are always globally generated. 
Recall that a line bundle $L$ on a projective variety $X$ is normally generated if it is very ample,
and the induced embedding of $X$ by the complete linear system $|L|$ is projectively normal.
If $P$ is a lattice polytope in $M_{\RR}$, and $X$, $L$ are the toric variety and the ample line bundle associated to
 $P$, then $L$ is normally generated if and only if for every $m\geq 2$ we have
 $$mP\cap M=((m-1)P\cap M)+(P\cap M).$$

\begin{proposition}
If $L$ is normally generated, then 
$\cM_{L^j}\otimes L^k$ is globally generated for all $j,k\geq1$.  
\end{proposition}
\begin{proof}
It suffices to show that 
$\cM_{L^j}\otimes L$ is 
globally generated. 
Let $P$ be the polytope associated to $L$ 
as in Proposition~\ref{prop:globgen}.
Suppose that $\sigma$ is a maximal cone, and let 
$w$ be a lattice point in  
$ jP \smallsetminus \{ju_{\sigma}\}$.
Since $L$ is normally generated, we can write $w=u+u'$ with 
$u\in (j-1)P\cap M$ and $u'\in P\cap M\smallsetminus \{u_{\sigma}\}$.
 Then $\{w,u+u_{\sigma}\}
\subseteq jP \cap (w+u_{\sigma} - P) 
 \cap M$, and now we apply
 Proposition~\ref{prop:globgen}.
\end{proof}

When $X$ is smooth, an even stronger statement holds. 
\begin{proposition}\label{prop:smooth}
Let $L,L'$ be ample  line bundles on a smooth toric variety 
$X$, and suppose that 
the multiplication map 
\[ H^0(X,L)\otimes H^0(X,L') \to H^0(X,L\otimes L')\]
is surjective.  Then $\cM_{L\otimes L'}\otimes L'$ is  
globally generated. 
\end{proposition}

\begin{proof}
Let $P$ and $P'$ be the polytopes associated to $L$ and $L'$, 
as in Proposition~\ref{prop:globgen}.
Suppose that $\sigma$ is a maximal cone, and let 
$w\in 
 \left( P+P'\right) \cap M  \setminus \{u_{\sigma}+u'_{\sigma}\}$.
 Since by assumption the map $P\cap M + P'\cap M \to \left(P+P'\right)\cap M$ 
 is surjective, we can write $w=u+u'$ for some $u\in P\cap M$ and $u'\in P'\cap M$. 
Since $X$ is smooth, we may assume in addition that $u'\neq u'_{\sigma}$.  
 It is now easy to see that $\{w,u+u'_{\sigma}\}\subseteq \left( P+P'\right) \cap (w+u'_{\sigma} - P') 
 \cap M$, and we conclude applying
 Proposition~\ref{prop:globgen}.
 \end{proof}

The following example shows that Proposition \ref{prop:smooth}
does not hold for arbitrary toric varieties. 

\begin{example}
Let $X, L$ be as in Example \ref{example11}. 
Then the multiplication map 
\[ H^0(X,L^2)\otimes H^0(X,L) \to H^0(X, L^3)\]
is surjective. However $\cM_{L^3}\otimes L$ is not 
globally generated. Indeed, we have
\[ 3P\cap \left( (1,1,1)  - P \right) = \{(1,1,1)\},\]
and so the condition of Proposition~\ref{prop:globgen} is not 
satified for $\sigma$ the cone corresponding to $(0,0,0)$, and 
$u=(1,1,1)\in 3P$.
\end{example}

\section{Restricting toric vector bundles to invariant curves}

We have shown that a toric vector bundle is nef or ample if and only
if its restriction to every invariant curve is nef or ample, and
that Seshadri constants of nef toric vector bundles can be computed
from restrictions to invariant curves, but we have so far avoided
the question of how to compute these restrictions.  In this section,
we show how to compute the splitting type of the restriction of a
toric vector bundle to an invariant curve, using continuous
interpolations of filtrations appearing in Klyachko's classification
of toric vector bundles. In order to do this, we review this
classification as well as the definition of the continuous
interpolations from \cite{Payne}. In this section, unless explicitly
mentioned otherwise, $X$ does not need to be complete. However, in
order to simplify some statements, we always assume that the maximal
cones in $\Delta$ have full dimension ${\rm rk}(N)$.

We systematically use the notation for equivariant line bundles
introduced at the beginning of \S 2.
For every cone $\sigma\in\Delta$, the restriction
$\cE\vert_{U_{\sigma}}$ decomposes as a direct sum of equivariant
line bundles $\cL_1\oplus\cdots\oplus\cL_r$. Moreover, each such
$\cL_i$ is equivariantly isomorphic to some
$\cL_{u_i}\vert_{U_{\sigma}}$, where the class of $u_i$ is uniquely
determined in $M/M\cap\sigma^{\perp}$. If $\sigma$ is a
top-dimensional cone, then in fact the multiset $\{u_1,\ldots,u_r\}$
is uniquely determined by $\cE$ and $\sigma$.

\smallskip

We now consider $T$-equivariant line bundles on invariant curves.

\begin{example}
Suppose that $X$ is complete, and let $C$ be the invariant curve in
$X$ associated to a codimension one cone $\tau$, with $\sigma$
and $\sigma'$ the maximal cones containing $\tau$.  Let $u$ and
$u'$ be linear functions in $M$ that agree on $\tau$.  Then we have
a $T$-equivariant line bundle $\cL_{u,u'}$ on the union $U_\sigma
\cup U_{\sigma'}$, obtained by gluing $\cL_{u}\vert_{U_{\sigma}}$
and $\cL_{u'}\vert_{U_{\sigma'}}$ using the transition function
$\chi^{u-u'}$, which is regular and invertible on $U_\tau$. In the
above we implicitly order $\sigma$ and $\sigma'$, but we hope that
this will not create any confusion.
 The
underlying line bundle of $\cL_{u,u'}\vert_C$ is $\cO(m)$, where $u
- u'$ is $m$ times the primitive generator of $\tau^\perp$ that is
positive on $\sigma$. It is easy to see that every equivariant line
bundle on $U_{\sigma} \cup U_{\sigma'}$ is equivariantly isomorphic
to some $\cL_{u,u'}$. Note that one can similarly define
$\cL_{u,u'}$ for \emph{any} two top-dimensional cones $\sigma$ and
$\sigma'$, if $u-u'\in (\sigma\cap\sigma')^{\perp}$.
\end{example}

\begin{lemma} \label{T line bundles}
With the notation in the above example, every $T$-equivariant line
bundle on $C=V(\tau)$ is equivariantly isomorphic to
$\cL_{u,u'}\vert_C$ for some unique pair of linear functions $u$ and
$u'$ that agree on $\tau$.
\end{lemma}

\begin{proof}
Suppose $\cL$ is a $T$-equivariant line bundle on $V(\tau)$.  On the
affine piece $C \cap U_\sigma$ we can choose an isotypical section
$s$ which is nonzero at the $T$-fixed point $x_\sigma$. Then the
locus where $s$ vanishes is closed, $T$-invariant, and does not
contain $x_\sigma$, and hence is empty. Therefore, $s$ gives an
equivariant trivialization that identifies $\cL$ with $\cO(\divisor
\chi^u)$ over $C \cap U_\sigma$, where $u$ is the isotypical type of
$s$. Similarly, there is an isotypical section $s'$ that identifies
$\cL$ with $\cO(\divisor \chi^{u'})$ over $C \cap U_{\sigma'}$. Then
there is a nonzero constant $c \in k^*$ such that $cs' = \chi^u
\cdot s$ over $U_\tau$, and it follows that $\cL$ is equivariantly
isomorphic to $\cL_{u, u'}\vert_C$.  Uniqueness follows from the
fact that $T$ acts on the fibers of $\cL$ over the fixed points
$x_\sigma$ and $x_{\sigma'}$ by the characters $\chi^u$ and
$\chi^{u'}$, respectively.
\end{proof}

\noindent  The lemma implies that the $T$-equivariant Picard group
of $V(\tau)$ is naturally isomorphic to the subgroup of $M \times M$
consisting of those pairs $(u,u')$ such that $u - u'$ vanishes on
$\tau$.

In order to describe the splitting type of the restriction of an
equivariant vector bundle $\cE$ on $X$ to an invariant curve, we
start by recalling the continuous interpolations of the filtrations
appearing in Klyachko's classification of toric vector bundles.
 See \cite{Klyachko} and \cite[Section~2]{Payne2} for proofs and further details.

Suppose that $X$ is a toric variety such that all maximal cones in
the fan have dimension ${\rm rk}(N)$. Given a toric vector bundle
$\cE$ on $X$, let $E$ denote the fiber of $\cE$ at the identity of
the torus $T$. For every cone $\sigma$ in $\Delta$, and for every
$u\in M$, evaluating sections at the identity gives an injective map
$\Gamma(U_{\sigma},\cE)_u\hookrightarrow E$. We denote by
$E^{\sigma}_u$ the image of this map.

Given $v\in |\Delta|$ and $t\in \RR$, we define a vector subspace
$E^v(t)\subseteq E$, as follows. Choose a cone $\sigma$ containing
$v$, and let
\[ E^v(t) = \sum_{\<u,v\> \geq t} E^\sigma_u.
\]
It is clear that for a fixed $v\in |\Delta|$, the vector subspaces
$\{E^v(t)\}_{t\in\RR}$ give a decreasing filtration of $E$ indexed
by real numbers.

This filtration can be described more explicitly as follows. Suppose
that $u_1,\ldots,u_r\in M$ are such that
$$\cE\vert_{U_{\sigma}}\simeq\bigoplus_{i=1}^r\cL_{u_i}\vert_{U_{\sigma}}.$$
If $L_i\subseteq E$ is the fiber of $\cL_{u_i}$ at the identity in
$T$, then we get a decomposition $E=L_1\oplus\cdots\oplus L_r$. With
this notation, $E^{\sigma}_u$ is spanned by those $L_i$ for which
$u_i-u\in \sigma^{\vee}$. It is easy to deduce from this that
\begin{equation}\label{description_E}
E^v(t)=\bigoplus_{\langle u_i,v\rangle\geq t}L_i.
\end{equation}
This description implies in particular that the definition of
$E^v(t)$ is independent of the choice of $\sigma$.

In addition, the above description shows that the filtrations we
have defined are ``piecewise-linear" with
 respect to $\Delta$, in the sense that for every maximal cone $\sigma$ in $\Delta$ there is a decomposition $E = \bigoplus_{u \in M} E_u$ such that
\[
E^v(t) = \bigoplus_{\<u, v\> \geq t} E_u
\]
for every $v \in \sigma$ and every real number $t$. Indeed, with the
notation in (\ref{description_E}), it is enough to take $E_u$ to be
the direct sum of the $L_j$ for which $u_j=u$.

\medskip

For every cone $\sigma$, if $u$, $u'\in M$ are such that
$u'\in\sigma^{\vee}$, multiplication by $\chi^{u'}$ induces an
inclusion
$$E^{\sigma}_u\subseteq E^{\sigma}_{u-u'}.$$
In particular, we have $E^{\sigma}_u=E^{\sigma}_{u-u'}$ if
$u'\in\sigma^{\perp}$. If $v_{\rho}$ is the primitive generator of a
ray $\rho$ in $\Delta$, and $i\in\ZZ$, we write $E^{\rho}(i)$ for
$E^{v_{\rho}}(i)$. It follows from the previous discussion that if
$u\in M$ is such that $\langle u,v_{\rho}\rangle =i$, then
$E^{\rho}(i)=E^{\rho}_u$.

\begin{KCT}
The category of toric vector bundles on $X(\Delta)$ is naturally
equivalent to the category of finite-dimensional $k$-vector spaces
$E$ with collections of decreasing filtrations
$\{E^\rho(i)\}_{i\in\ZZ}$ parametrized by the rays in $\Delta$, and
satisfying the following compatibility condition: For each maximal
cone $\sigma \in \Delta$, there is a decomposition $E = \bigoplus_{u
\in M} E_{u}$ such that
\[
E^\rho(i) = \sum_{\<u, v_\rho\> \, \geq \, i} E_{u},
\]
 for every ray $\rho$ of $\sigma$ and $i \in \ZZ$.
\end{KCT}

Of course, the equivalence of categories is given by associating to
a toric vector bundle $\cE$ its fiber $E$ over the identity in the
dense torus, with filtrations $E^\rho(i)$ as above. Note that the
filtrations $\{E^v(t)\}_{t\in\RR}$ give continuous interpolations of
the filtrations $\{E^{\rho}(i)\}_{i\in\ZZ}$. They were introduced in
\cite{Payne} to study
 equivariant vector bundles with trivial ordinary total Chern class.

\begin{remark}\label{decomposable}
Let $\cE$ be an equivariant vector bundle on $X$, and
$\{E^{\rho}(i)\}$ the filtrations that appear in the above theorem.
It is easy to see that $\cE$ is equivariantly isomorphic to a direct
sum of equivariant line bundles if and only if there is a
decomposition $E=L_1\oplus\cdots\oplus L_r$, with each $L_j$ a
one-dimensional subspace, and such that each $E^{\rho}(i)$ is a
direct sum of some of the $L_j$. Of course, the $L_j$ are the fibers
of the corresponding line bundles at the identity of $T$.
\end{remark}

We mention one continuity result for these filtrations that we will
need \cite[Lem\-ma~4.7]{Payne}. On the set
\[
\coprod_{\ell} \Gr(\ell,E),
\]
of subspaces of $E$,
 partially ordered by inclusion, consider the poset topology.
A subset  $S \subset \coprod_{\ell} \Gr(\ell,E)$ is closed in this
topology if and only if every subspace of $E$ that contains an
element of $S$ is also in $S$.
 The map taking a point $v \in |\Delta|$ and a real number $t$ to
$E^v(t)$ is a continuous map from $|\Delta| \times \RR$ to
$\coprod_{\ell} \Gr(\ell,E)$.

Evaluation at the identity gives a canonical isomorphism
\[
\Gamma(X, \cE)_u
=\bigcap_{\sigma\in\Delta}\Gamma(U_{\sigma},\cE)_u\xrightarrow{\sim}
\bigcap_{v \in |\Delta|} E^v(\<u,v\>).
\]
The infinite intersection $\bigcap_{v} E^v(\<u,v\>)$ is the same as
the finite intersection $\bigcap_\rho E^\rho(\<u, v_\rho \>)$, but
the advantage of working with the $\RR$-graded interpolations is
that it allows us to use continuity and convexity to find global
sections, generalizing standard convexity arguments for constructing
isotypical global sections of toric line bundles, as in
\cite[Section~3.4]{Fulton}.
 It also facilitates the computation of the restriction of $\cE$ to an invariant curve, as we will see below.

\bigskip

From now on we assume that $X$ is complete. Our next goal is to
describe the equivariant vector bundles on the invariant curves in
$X$. Recall that a well-known result due to Grothendieck says that
every vector bundle on $\PP^1$ splits as a sum of line bundles. It
does not follow tautologically that every $T$-equivariant vector
bundle splits equivariantly as a
 sum of line bundles, but this has been shown by Kumar \cite{Kumar} and may also be deduced from \cite[Example~2.3.3 and Section~6.3]{Klyachko}.
For the reader's convenience, we give below a direct argument based on
Klyachko's Classification Theorem.  We start with the following lemma.

\begin{lemma}\label{intersection_Borel}
Given two flags $\Fl$ and $\Fl'$ of subspaces in a finite
dimensional vector space $V$, there is a decomposition
$V=L_1\oplus\cdots\oplus L_r$, with all $L_i$ one-dimensional, such
that every subspace in either $\Fl$ or $\Fl'$ is a direct sum of
some of the $L_i$.
\end{lemma}

\begin{proof}
After refining the two flags, we may assume that both $\Fl$ and
$\Fl'$ are complete flags. Recall that the intersection of two Borel
subgroups in a linear algebraic group contains a maximal torus (see
\cite{DM}, Cor. 1.5). Therefore the intersection of the stabilizers
of the two flags contains the stabilizer of a collection of
one-dimensional subspaces $L_1,\ldots,L_r$. These subspaces satisfy
our requirement.
\end{proof}

\begin{corollary}\label{P1}
Let $X$ be a complete toric variety. Any $T$-equivariant vector
bundle $\cE$ on the invariant curve $C=V(\tau)$ splits equivariantly
as a sum of line bundles
\[
\cE \cong \cL_{u_1, u'_1}\vert_C \oplus \cdots \oplus \cL_{u_r,
u'_r}\vert_C.
\]
\end{corollary}

\begin{proof}
By Lemma~\ref{T line bundles}, it is enough to show that $\cE$
decomposes as a direct sum of $T$-equivariant line bundles. If
$X=\PP^1$, this is clear: it is enough to apply the criterion in
Remark~\ref{decomposable} and Lemma~\ref{intersection_Borel}. The
general case reduces easily to this one: the exact sequence
$$0\to M\cap\tau^{\perp}\to M\to M/M\cap\tau^{\perp}\to 0$$
induces an exact sequence of tori
$$0\to T'\to T\to T_C\to 0,$$
where $T_C$ is the dense torus in $C$, and $T'={\rm
Spec}(k[M/M\cap\tau^{\perp}])$. We choose a splitting of $T\to T_C$.
Since $T'$ acts trivially on $C$, it follows that we have a
decomposition $\cE=\cE_1\oplus\cdots\oplus\cE_m$, where each $\cE_i$
is a $T_C$-equivariant bundle, and $T'$ acts on $\cE_i$ via a character
$\chi_i$. We can decompose each $\cE_i$ as a direct sum of
$T_C$-equivariant line bundles, and each of these is, in fact, a
$T$-equivariant subbundle of $\cE$.
\end{proof}

We will see below that the pairs $(u_i,u'_i)$ that appear in the
above corollary are uniquely determined by $\cE$ (up to reordering).
We first give a condition for an analogue of the corollary to hold
in a suitable neighborhood of the invariant curve. More precisely,
let $V(\tau)$ be the invariant curve corresponding to a codimension
one cone $\tau$ in $\Delta$, and let $\sigma$ and $\sigma'$ be the
maximal cones containing $\tau$.  If the restriction of $\cE$ to
$U_\sigma \cup U_{\sigma'}$ splits as a certain sum of line bundles
$\cL_{u,u'}$, then the restriction of $\cE$ to $V(\tau)$ has the
same splitting type, tautologically.  But the restriction of a toric
vector bundle to $U_\sigma \cup U_{\sigma'}$ need not split as a sum
of line bundles, even for a rank two bundle on a toric surface.

\begin{example}[Tangent bundle on $\PP^2 \smallsetminus \mathrm{pt}$]
Let $\sigma_1$ and $\sigma_2$ be two maximal cones in the fan
defining $X=\PP^2$, and let $\cE=T_X$ be the tangent bundle of $X$.
If $U=U_{\sigma_1}\cup U_{\sigma_2}$, then $\cE\vert_U$ does not
split as a sum of line bundles, even if we disregard the equivariant
structure. Indeed, note first that since the complement of $U$ is a
point, it has codimension two in $\PP^2$. In particular,
$\Pic(U)\simeq\Pic(\PP^2)$, and for every vector bundle $\cF$ on
$\PP^2$, we have $\Gamma(\PP^2,\cF)\simeq\Gamma(U,\cF)$.  If $\cE\vert_U$ is decomposable, then it has to be isomorphic to
$\cO(a)\oplus\cO(b)\vert_U$. Restricting to a line that is contained
in $U$, we then see that we may take $a=2$ and $b=1$. On the other
hand
$$h^0(U,\cO(2)\oplus\cO(1))=h^0(\PP^2,\cO(2)\oplus\cO(1))=9,$$
while $h^0(U,\cE)=h^0(\PP^2,T_{\PP^2})=8$, a contradiction.
\end{example}

\noindent  However, we have the following combinatorial condition
that guarantees the restriction of a toric vector bundle to the
union of two invariant affine open subvarieties splits. Given the
equivariant vector bundle $\cE$ on $X$ and $v \in N_\RR$, let
$\Fl(v)$ be the partial flag of subspaces of $E$ appearing in the
filtration $E^v(t)$.

\begin{proposition} \label{two affines}
Let $\sigma$ and $\sigma'$ be two maximal cones in $\Delta$. If
$\Fl(v)$ is constant on the relative interiors of $\sigma$ and
$\sigma'$, then the restriction of $\cE$ to $U_\sigma \cup
U_{\sigma'}$ splits equivariantly as a sum of line bundles.
\end{proposition}

\begin{proof}
Let $\Fl$ and $\Fl'$ be the partial flags in $E$ associated to
points in the relative interiors of $\sigma$ and $\sigma'$,
respectively. It follows from Lemma~\ref{intersection_Borel} that
there is a splitting $E = L_1 \oplus \cdots \oplus L_r$ such that
every subspace appearing in
 $\Fl$ and $\Fl'$ is a sum of some of the $L_i$.

Now, for any ray $\rho$ of $\sigma$, every subspace appearing in the
filtration $E^\rho(i)$ is in $\Fl$, by the continuity of the
interpolations \cite[Lemma~4.7]{Payne}.  Similarly, if $\rho'$ is a
ray of $\sigma'$ then every subspace $E^{\rho'}(i)$ is in $\Fl'$. In
particular, each of these subspaces is the sum of some of the $L_j$,
hence by Remark~\ref{decomposable} we deduce that the restriction of
$\cE$ to $U_{\sigma}\cup U_{\sigma'}$ splits as a sum of equivariant
line bundles.
\end{proof}

With $\cE$ fixed, there is a canonical coarsest subdivision
$\Delta'$ of $\Delta$ such that $\Fl(v)$ is constant on the relative
interior of each maximal cone in $\Delta'$, as follows.  Suppose
$\sigma$ is a maximal cone in $\Delta$, and let $u_1, \ldots, u_r$
in $M$ be such that
$$\cE\vert_{U_{\sigma}}\simeq\bigoplus_{i=1}^r\cL_{u_i}\vert_{U_{\sigma}}.$$
It follows from (\ref{description_E}) that if $L_i\subseteq E$ is
the subspace corresponding to $\cL_{u_i}$, then
 $$E^v(t) = \bigoplus_{\<u_i, v \> \geq t } L_i$$ for every $v$ in $\sigma$ and $t$ in $\RR$.
Hence $\Fl(v)$ is constant on the interior of a top dimensional cone
contained in $\sigma$ if and only if this interior does not meet any
of the hyperplanes $(u_i-u_j)^{\perp}$, with $u_i\neq u_j$.
Therefore the maximal cones of $\Delta'$ are
 exactly the closures of all chambers of $\sigma$ lying in the complement of the above hyperplane arrangement,
 for all maximal cones $\sigma$ in $\Delta$.

Note that we have a proper birational toric morphism $p: X'
\rightarrow X$
 associated to this subdivision, where $X' = X(\Delta')$, and the restriction of $p^* \cE$ to any union of two invariant affine open subvarieties splits
  equivariantly as a sum of line bundles $\cL_{u,u'}$.  For any invariant curve $C$ in $X$, we can choose an invariant
  curve $C'$ in $X'$ projecting isomorphically onto $C$, and the splitting type of $\cE|_C$ is tautologically the same
  as the splitting type of $p^* \cE |_{C'}$, which we can compute more easily.

\medskip

We now assume that $\cE$ and $V(\tau)$ satisfy the hypothesis of
Proposition~\ref{two affines}. By the proposition, we can find a
multiset $\bu_C \subset M \times M$ such that
\[
\cE\vert_U \simeq \bigoplus_{(u,u') \in \bu_C} \cL_{u,u'},
\]
where $U=U_{\sigma}\cup U_{\sigma'}$. The following lemma relates
$\bu_C$ to the filtrations on $E$ corresponding to $\sigma$ and
$\sigma'$.

\begin{lemma} \label{intersections}
Let $C$ be the invariant curve corresponding to the intersection of two adjacent maximal cones $\sigma$ and $\sigma'$. If $\Fl(v)$ is constant on the
interiors of $\sigma$ and $\sigma'$, then
\begin{equation}\label{formula 20}
\dim \big( E^v(t) \cap E^{v'}(t') \big) = \# \{(u,u') \in \bu_C \ \mid \ \<u,v \> \geq t \mbox{ and } \<u',v' \> \geq t' \}
\end{equation}
for any $v \in \sigma$ and $v' \in \sigma'$, and for any real
numbers $t$ and $t'$.
\end{lemma}

\begin{proof}
By Proposition~\ref{two affines}, the restriction of $\cE$ to
$U_\sigma \cup U_{\sigma'}$ splits as a sum of line bundles
 $\cE \simeq \cL_{u_1, u'_1} \oplus \cdots \oplus \cL_{u_r, u'_r}$, where $\bu_C = \{(u_1, u'_1), \ldots, (u_r, u'_r)\}$
 (note that we might have repetitions).
   Let $L_i$ be the fiber of the subbundle $\cL_{u_i, u'_i}$ over the identity.  Then
\[
E^v(t) = \bigoplus_{\<u_i,v\> \geq t} L_i \mbox{ \ \ and \ \ }
E^{v'}(t') = \bigoplus_{\<u'_j,v'\> \geq t'} L_j.
\]
Therefore $E^v(t) \bigcap E^{v'}(t')$ is the sum of those $L_i$ such
that $\<u_i,v\> \geq t$ and $\<u'_i, v' \> \geq t'$. The assertion
in the lemma follows.
\end{proof}

\begin{remark}\label{uniqueness}
Note that the pairs $(u_1,u'_1),\ldots,(u_r,u'_r)$ such that
$$\cE\vert_{U_{\sigma}\cup U_{\sigma'}}\simeq\bigoplus_{i=1}^r\cL_{u_i,u'_i}$$
are unique, up to reordering. This is an easy consequence of
equation (\ref{formula 20}), since the left-hand side of the formula
does not depend on the choice of the pairs $(u_i,u'_i)$. We can
deduce from this also the uniqueness of the decomposition in
Corollary~\ref{P1}.
\end{remark}

\begin{corollary}\label{P2}
If $X$ is a complete toric variety, and if $\cE$ is a
$T$-equivariant vector bundle on the invariant curve $C=V(\tau)$,
then the pairs $(u_i,u'_i)$ such that
\[
\cE \cong \cL_{u_1, u'_1}\vert_C \oplus \cdots \oplus \cL_{u_r,
u'_r}\vert_C
\]
are unique, up to reordering.
\end{corollary}

\begin{proof}
We may argue as in the proof of Corollary~\ref{P1} to reduce to the
case when $X=\PP^1$. In this case the hypothesis in
Lemma~\ref{intersections} is clearly satisfied, and we get our
assertion as in the previous remark.
\end{proof}

\section{Sections of nef toric vector bundles and a triviality criterion}

It is well-known that every nef line bundle on a complete toric
variety is globally generated. As Examples~\ref{example11}, \ref{example12}
and \ref{example13} show, we cannot expect the same result to hold in higher rank.
The correct generalization in higher rank is given 
by the following theorem.

\begin{theorem}\label{sections}
If $\cE$ is a nef toric vector bundle on the complete toric variety
$X$, then for every point $x\in X$ there is a section of $\cE$ that
does not vanish at $x$.
\end{theorem}

\begin{proof}
We will systematically use the notation introduced in \S 5. Note
that if $f\colon X'\to X$ is a proper, birational toric morphism,
then we may replace $X$ and $\cE$ by $X'$ and $f^*\cE$. Indeed,
$f^*\cE$ is nef, and we have an isomorphism
$\Gamma(X,\cE)\simeq\Gamma(X',f^*\cE)$. Hence in order to find a
section of $\cE$ that does not vanish at some $x\in X$, it is enough
to find a section $s'$ of $f^*\cE$ that does not vanish at some
point in the fiber $f^{-1}(x)$. We deduce that after subdividing
$\Delta$, we may assume that $\Fl(v)$ is constant on the interior of
each maximal cone.

Since the space of global sections of $\cE$ has a basis of
$T$-eigensections, the subset of $X$ where all global sections
vanish is closed and $T$-invariant.  Therefore, it will suffice to
prove that $\cE$ has a nonvanishing global section at every fixed
point. Let $x = x_\sigma$ be the fixed point corresponding to a
maximal cone $\sigma$.

Consider $u_1, \ldots, u_r$ in $M$ such that
$$\cE\vert_{U_{\sigma}}
\simeq\bigoplus_{i=1}^r\cL_{u_i}\vert_{U_{\sigma}}.$$ We have seen
that if $L_i\subseteq E$ is the fiber of $\cL_{u_i}$ at the identity
of $T$ then $$E^v(t) = \bigoplus_{\<u_i, v\> \geq t} L_i$$ for all
$v$ in $\sigma$ and $t \in \RR$. By assumption, if $u_i\neq u_j$,
then $u_i-u_j$ does not vanish on the interior of $\sigma$. After
reordering, we may assume $u_1 \geq \cdots \geq u_r$ on $\sigma$.
Let $u = u_1$ and $L = L_1$. There is a $\chi^u$-isotypical section
$s$ of $\cE$ over $U_{\sigma}$ that is nonvanishing at $x$, and
whose value at the identity spans $L$.  We claim that $s$ extends to
a regular section of $\cE$ over all of $X$.
 To prove the claim, we must show that $E^v(\<u,v\>)$ contains $L$ for every $v \in N_\RR$.

 After further subdividing
 $\Delta$,
 we may also assume that for every maximal cone $\sigma'$,
if
$$\cE\vert_{U_{\sigma'}}\simeq\bigoplus_{i=1}^r\cL_{u'_i}\vert_{U_{\sigma'}},$$
then for every $i$ such that $u\neq u'_i$, the linear function
$u-u'_i$ does not vanish on the interior of $\sigma'$.

Let $v$ be a point in $N_\RR$.  Choose $v_0$ in the interior of
$\sigma$ such that the segment $S=[v_0,v)$ is disjoint from the
codimension two cones in $\Delta$, so $[v_0, v)$ passes through a
sequence of maximal cones
\[
\sigma = \sigma_0, \sigma_1, \ldots, \sigma_s
\]
such that $\sigma_{j-1}$ and $\sigma_{j}$ intersect in codimension
one, for $1 \leq j \leq s$.
 Let $v_j$ be a point in the interior of $\sigma_j$, and let $\tau_j = \sigma_j\cap\sigma_{j+1}$.
It is enough to show that $E^{v_s}(\<u,v_s\>)$ contains $L$. Indeed,
if this holds for every $v_s$ in the interior of $\sigma_s$ then we
conclude by the continuity of the interpolation filtrations that
$L\subseteq E^v(\<u,v\>)$ as well. Since $L\subseteq
E^{v_0}(\<u,v_0\>)$ by construction, we see that in order to
conclude it is enough to show that $E^{v_j}(\<u,v_j\>)\subseteq
E^{v_{j+1}}(\<u,v_{j+1}\>)$, for every $0\leq j\leq s-1$.

Let $u_{i,j}\in M$ be such that for $0\leq j\leq s$ we have
$$\cE\vert_{U_{\sigma_j}}\simeq\bigoplus_{i=1}^r\cL_{-u_{i,j}}\vert_{U_{\sigma_j}}.$$
Note that for every $t$ we have
\begin{equation}\label{formula 30}
\dim_k E^{v_j}(t)=\#\{i\mid\<u_{i,j},v_j\>\geq t\}.
\end{equation}
After reordering the $u_{i,j}$, we may assume that $u_{i,0}=u_i$ for
every $i$, and that we have
$$\cE\vert_{U_{\sigma_j}\cup
U_{\sigma_{j+1}}}\simeq\bigoplus_{i=1}^r\cL_{u_{i,j},u_{i,j+1}}$$
for $j\leq s-1$.

We denote by $\Psi_i$ the piecewise linear function on the segment
$S$, that is given on $S\cap\sigma_j$ by $u_{i,j}$. Since
$\cE\vert_{V(\tau_j)}$ is nef, we deduce that $u_{i,j}\geq
u_{i,j+1}$ on $\sigma_j$. This implies that each $\Psi_i$ is convex.

For every $j\in\{0,\ldots,s\}$, let $I_j$ be the set of those $i\leq
r$ such that $\Psi_i\geq u$ on $\sigma_j$. Since
$\<u,v_0\>\geq\Psi_i(v_0)$ for every $i$, and since $\Psi_i$ is
convex, it follows that $I_j\subseteq I_{j+1}$ for $j\leq s-1$. By
assumption, $\<u_{i,j},v_j\>\geq \<u,v_j\>$ if and only if
$u_{i,j}-u$ lies in $\sigma_j^{\vee}$. Therefore (\ref{formula 30})
implies that
$$\# I_j=\dim_kE^{v_j}(\<u,v_j\>).$$
On the other hand, Lemma~\ref{intersections} gives
$$\dim_k\left(E^{v_j}(\<u,v_j\>)\cap E^{v_{j+1}}(\<u,v_{j+1}\>)\right)=\#(I_j\cap I_{j+1})=\# I_j.$$
We conclude that $E^{v_j}(\<u,v_j\>)\subseteq
E^{v_{j+1}}(\<u,v_{j+1}\>)$ for $0\leq j\leq s-1$ and, as we have
seen, this completes the proof.
\end{proof}

\bigskip
\begin{corollary}\label{cor3}
Let $\cE$ be a toric vector bundle on the smooth complete toric
variety $X$.
\begin{enumerate}
\item[i)] Suppose that $\cE$ is ample, and let $x\in X$ be a torus-fixed
point. If $y\neq x$ is another point in $X$ then there is a section $s\in
H^0(X,\cE)$ such that $s(x)=0$ and $s(y)\neq 0$. Moreover, for every
nonzero $v\in T_xX$, there is $s\in H^0(X,\cE)$ such that $x$ lies
in the zero-locus $Z(s)$ of $s$, but $v\not\in T_xZ(s)$.
\item[ii)] If the Seshadri constant $\epsilon(\cE)$ is at least two then, for every distinct torus-fixed points
$x_1,\ldots,x_k$ on $X$, and for every $x$ that is different from
all $x_i$, there is $s\in H^0(X,\cE)$ such that $s(x_i)=0$ for all
$i$, and $s(x)\neq 0$. In particular, $h^0(\cE)\geq k+1$.
\end{enumerate}
\end{corollary}

\begin{proof}
For i), let $p\colon \widetilde{X}\to X$ be the blow-up of $x$, with
exceptional divisor $F$. We deduce from Corollary~\ref{cor1} that
$p^*(\cE)\otimes \cO(-F)$ is nef, and both assertions follow from
Theorem~\ref{sections} since
$$H^0(\widetilde{X},p^*(\cE)\otimes \cO(-F))\simeq \left\{s\in H^0(X,\cE)\mid
s(x)=0\right\}.$$ Similarly, ii) follows from
Corollary~\ref{cor1'} and Theorem~\ref{sections}.
\end{proof}

\begin{remark}
Note that in rank one, the property in Corollary~\ref{cor3} i) says that $\cE$ separates
points and tangent vectors, that is, $\cE$ is very ample. In higher rank, however, the property
is weaker than the very ampleness of $\cE$ (which by definition means the very
ampleness of $\cO_{\PP(\cE)}(1)$).
\end{remark}

\bigskip

Our next goal is to prove the characterization of trivial toric
vector bundles. This is the toric generalization of the assertion
that a vector bundle on $\PP^n$ is trivial if and only if its
restriction to every line is trivial (see \cite[Theorem 3.2.1]{OSS}). The result answers affirmatively a question posed by
V.~Shokurov. We call a toric vector bundle \emph{trivial} if it is
isomorphic to $\cO^{\oplus r}$ disregarding the equivariant
structure.

\begin{theorem}\label{trivial}
Let $\cE$ be a toric vector bundle on the complete toric variety
$X$. Then $\cE$ is trivial if and only if its restriction to every
irreducible invariant curve on $X$ is trivial.
\end{theorem}

\begin{proof}
We prove the assertion by induction on $r={\rm rk}(\cE)$, the case
$r=0$ being vacuous. Note first that since the restriction of $\cE$
to every invariant curve is trivial, in particular nef, it follows
from Theorem~\ref{char_nef} that $\cE$ is nef. Therefore
Theorem~\ref{sections} implies that for every point $x\in X$, there
is a a section $s\in \Gamma(X,\cE)$ that does not vanish at $x$. Fix
a maximal cone $\sigma\in X$, and choose a section $s_0$ that does
not vanish at $x_{\sigma}$. We may assume that $s_0$ is
$\chi^u$-isotypical for some $u\in M$.

\noindent\begin{claim} If $s$ is an isotypical section of $\cE$ that
does not vanish at a fixed point $x_{\sigma}$ then $s$ is nowhere
vanishing.
\end{claim}

The claim implies that our section $s_0$ gives a trivial equivariant
subbundle $\cL$ of $\cE$, with ${\rm rk}(\cL)=1$. The restriction of
the quotient $\cE/\cL$ to any invariant curve is trivial, hence the
inductive assumption implies that $\cE/\cL$ is trivial. Since
$H^1(X,\cO_X)=0$, it follows that the exact sequence
\[
0 \rightarrow \cL \rightarrow \cE \rightarrow \cE/ \cL \rightarrow 0
\]
splits (non-equivariantly),  hence $\cE$ is trivial. Therefore, it is
enough to prove the claim.

If  we have a proper, birational toric morphism 
$f\colon X'\to X$ then it is
enough to prove the claim for the section
$f^*(s)\in\Gamma(X',f^*\cE)$. After subdividing $\Delta$, we may
assume that $\Fl(v)$ is constant on the interior of every maximal
cone $\sigma$. Proposition~\ref{two affines} implies that given two
maximal cones $\sigma_1$ and $\sigma_2$ whose intersection $\tau$
has codimension one, if $U=U_{\sigma_1}\cup U_{\sigma_2}$, then
$$\cE\vert_U\simeq\cL_{u_1,u'_1}\oplus\cdots\oplus\cL_{u_r,u'_r}$$ for suitable
$u_i$, $u'_i\in M$. Since the restriction $\cE\vert_{V(\tau)}$ is
trivial, it follows that $u_i=u'_i$ for every $i$. Therefore the
restriction $\cE\vert_U$ is trivial (disregarding the equivariant
structure).

Suppose $s\in\Gamma(X, \cE)$ is an isotypical section that does not
vanish at $x_{\sigma_1}$.  Then the restriction  $s\vert_{U_{\sigma_1}}$
corresponds to $(\phi_1,\ldots,\phi_r)$, where one of the $\phi_i$ is a nonzero constant function, via the isomorphism $\cE\vert_U\simeq
\cO_U^{\oplus r}$. The analogous assertion then holds for
$s\vert_U$, and we conclude that $s$ does not vanish at
$x_{\sigma_2}$. This implies the claim, and the theorem follows.
\end{proof}

\section{Open questions}

In this section we list several open questions. The first questions
are motivated by the corresponding results in the rank one case. It
is likely that the situation in higher rank is more complicated, but
it would be desirable to have explicit examples to illustrate the
pathologies in rank $>1$.

\begin{question}\label{q3}
Suppose that $\cE$ is a toric vector bundle on the complete toric
variety $X$. Is the $k$-algebra
\begin{equation}\label{algebra_of_E}
\bigoplus_{m\geq 0}H^0(X, {\rm Sym}^m(\cE))
\end{equation}
finitely generated?
\end{question}

For the corresponding assertion in the case of line bundles, see for
example \cite{Eli}. Note that if $f\colon Y\to X$ is a toric resolution of singularities 
of $X$, then the projection formula implies that the $k$-algebra corresponding to $\cE$
is isomorphic to the $k$-algebra corresponding to $f^*(\cE)$. Therefore it is enough to 
consider Question~\ref{q3} when $X$ is smooth.

One can ask whether  $\PP(\cE)$ satisfies the following stronger property.
Suppose that $Y$ is a complete variety such that $\Pic(Y)$ is finitely generated 
(note that every projective bundle over a toric variety has this property). Following 
\cite{HK} we say that $Y$ is a Mori Dream
Space if for every finite set of line bundles
$L_1,\ldots,L_r$ on $Y$, the $k$-algebra
\begin{equation}\label{def_Cox1}
\bigoplus_{m_1,\ldots,m_r\geq 0}H^0(Y, L_1^{m_1}\otimes\cdots\otimes
L_r^{m_r})
\end{equation} 
is finitely generated. 
(The definition in \emph{loc. cit.}
 requires $X$ to be $\QQ$-factorial, but this condition is not relevant for us.)
 Equivalently, it is enough to put the condition that the $k$-algebra
 \begin{equation}\label{def_Cox2}
\bigoplus_{m_1,\ldots,m_r\in\ZZ}H^0(Y,L_1^{m_1}\otimes\cdots\otimes
L_r^{m_r})
\end{equation} 
is finitely generated, when $L_1,\ldots,L_r$ generate $\Pic(X)$ as a group
(if they generate $\Pic(X)$ as a semigroup, then it is enough to let $m_1,\ldots,m_r$
vary over $\NN$).

It is well-known that a complete toric variety is a Mori Dream Space. This follows using the fact
that in this case the $k$-algebra (\ref{def_Cox1}) is isomorphic to
$\bigoplus_{m\geq 0} H^0(\PP,\cO(m))$, where $\PP=\PP(L_1\oplus\cdots\oplus L_m)$
is again a toric variety. In fact, if $X$ is smooth, and if $L_1,\ldots,L_r$ form a basis of
$\Pic(X)$, then the $k$-algebra (\ref{def_Cox2}) is a polynomial ring, the 
\emph{homogeneous coordinate ring} of $X$
(see \cite{Cox}).

\begin{question}\label{q4}
If $X$ is a complete toric variety, and if $\cE$ is a
toric vector bundle on $X$, is $\PP(\cE)$ a Mori Dream Space?
Since $\Pic(\PP(\cE))$ is generated by $\cO(1)$ and by the pull-backs of the line bundles 
on $X$, this can be restated as follows: if
$L_1,\ldots,L_r$ are arbitrary line bundles on $X$, is the 
$k$-algebra
\begin{equation}\label{MDS}
\bigoplus_{m,m_1,\ldots,m_r\geq 0}H^0(X,{\rm Sym}^m(\cE)\otimes
L_1^{m_1}\otimes\cdots\otimes L_r^{m_r})
\end{equation}
 finitely generated?
 \end{question}
 
\noindent Note that Remark~\ref{rem_nef_cone} gives some positive evidence in the
direction of this question (it is a general fact that the Mori cone of a Mori Dream
Space is rational polyhedral).

\begin{remark}\label{equivalent}
It is clear that a positive answer to Question~\ref{q4} implies a positive answer to Question~\ref{q3}.
However, the converse is also true: if the $k$-algebra (\ref{algebra_of_E})
is finitely generated for every toric vector bundle, then Question~\ref{q4} has a positive answer.
Indeed, given a toric vector bundle $\cE$ on $X$, and $L_1,\ldots,L_r\in\Pic(X)$, let
$\cE'=\cE\oplus L_1\oplus\cdots\oplus L_r$. Since
$$\bigoplus_{m,m_1,\ldots,m_r\geq 0}H^0(X,{\rm Sym}^m(\cE)\otimes
L_1^{m_1}\otimes\cdots\otimes L_r^{m_r})
\simeq\bigoplus_{m\geq 0}H^0(X, {\rm Sym}^m(\cE')),$$
we see that this is a finitely generated $k$-algebra.
\end{remark}

\begin{remark}
In connection with Question~\ref{q4}, note the following module-theoretic finiteness
statement. 
 Suppose that $X$ is a smooth toric variety, and that $\cE$ is a reflexive sheaf
on $X$, not necessarily equivariant.
Choose a basis $L_1,\ldots,L_r$ for $\Pic(X)$, and let $S$ 
be the homogeneous coordinate ring of $X$. If 
$$M=\bigoplus_{m_1,\ldots,m_r\in\ZZ}H^0(X,\cE\otimes L_1^{m_1}\otimes\cdots\otimes
L_r^{m_r}),$$
then $M$ is finitely generated as an $S$-module.  This follows using some basic facts about the homogeneous coordinate ring for which we refer to \cite{Cox}. Indeed, via the correspondence
between graded $S$-modules and coherent sheaves, one can express every coherent sheaf on 
$X$ as a quotient of a direct sum of line bundles on $X$. If we write $\cE^{\vee}$ as the quotient 
of $L'_1\oplus\cdots\oplus L'_s$, where $L'_i=L_1^{m_{i,1}}\otimes\cdots\otimes L_r^{m_{i,r}}$,
we see by taking duals that $M$ is embedded in the free module $S^{\oplus s}$. Since $S$ is Noetherian, it follows that $M$ is a finitely generated $S$-module.
\end{remark}

\bigskip

Recall that a vector bundle $\cE$ is \emph{very ample} if the line
bundle $\cO(1)$ on $\PP(\cE)$ is very ample. In light of Examples~\ref{example11}, \ref{example12}
and \ref{example13}, it seems unlikely that the following question would have a positive answer,
but it would be nice to have explicit counterexamples.

\begin{question}\label{q2}
Let $\cE$ be an ample toric vector bundle on a smooth complete toric
variety $X$. Is $\cE$ very ample? Is $\cE$ globally generated?
\end{question}

 Examples~\ref{non_vanishing} and \ref{Fujino} show that the higher cohomology groups of an ample toric vector bundle
 $\cE$
 do not vanish in general. The cohomology of $\cE$ is canonically identified with the cohomology
 of the ample line bundle $\cO(1)$ on $\PP(\cE)$, hence some projectivized toric vector bundles have ample line bundles with non-vanishing cohomology. It should be interesting to find conditions on 
 $\cE$ that guarantee that  ample line bundles on $\PP(\cE)$ have no higher cohomology. In characteristic 
 $p$, a condition that guarantees such vanishing is the splitting of the Frobenius morphism on
 $\PP(\cE)$ (we refer to \cite{BK} for basic facts about Frobenius split varieties). 

\begin{question}
Let $\cE$ be a toric vector bundle on a complete toric variety $X$ over a field
of positive characteristic. When is $\PP(\cE)$ Frobenius split?
\end{question}

\noindent We point out that this condition is independent of twisting $\cE$ by a line bundle. 
Note also that every toric variety is Frobenius split. In particular, $\PP(\cE)$ is Frobenius split
if $\cE$ decomposes as a sum of line bundles. Furthermore, some indecomposable toric
vector bundles have Frobenius split projectivizations: for example, the tangent bundle on the projective
space $\PP^n$. Indeed,
 we have 
 $$\PP(T_{\PP^n})\simeq\{(q,H)\in\PP^n\times (\PP^n)^*\mid q\in H\},$$
 hence $\PP(T_{\PP^n})$ is a homogeneous variety, and all homogeneous varieties are Frobenius split.

\bigskip

It would be interesting to find a criterion for a toric vector
bundle to be big. We propose the following.

\begin{question}\label{q5}
Let $\cE$ be a toric vector bundle on the projective toric variety
$X$. Is it true that $\cE$ is big if and only if for every morphism
$f\colon\PP^1\to X$ whose image intersects the torus, we have
$f^*(\cE)$ big?
\end{question}

It is easy to see that if $\cE$ is a big toric vector bundle on the
projective toric variety $X$, and if $f\colon C\to X$ is a
non-constant morphism from a projective curve $C$ such that ${\rm
Im}(f)$ intersects the torus, then $f^*(\cE)$ is big. Indeed, one
can use the fact that a vector bundle $\cE$ is big if and only if for some
(every) ample line bundle $A$ we have
$$H^0(X,{\rm Sym}^m(\cE)\otimes A^{-1})\neq 0$$
for some $m>0$ \cite[Example~6.1.23]{positivity}. In our
situation, fix an ample line bundle $A$ on $X$, and consider on $A$
an equivariant structure. Let $m>0$ be such that there is a nonzero
$s\in H^0(X,{\rm Sym}^m(\cE)\otimes A^{-1})$. It is clear that we
may assume that $s$ is an eigenvalue for the corresponding torus
action. In particular, the zero-set $Z(s)$ is contained in the
complement of the torus. Therefore $f^*(s)$ gives a nonzero section
in $H^0(C,f^*({\rm Sym}^m(\cE))\otimes f^*(A)^{-1})$. Since $f^*(A)$ is ample on
$C$, it follows that $f^*(\cE)$ is big.

For line bundles it is known that the converse is also true (see \S
3 in \cite{Payne3}). This is the toric analogue of a result from
\cite{BDPP} describing the big cone as the dual of the cone of
movable curves.

\bigskip

In light of Corollary~\ref{cor2}, we consider the following

\begin{question}\label{q6}
Let $L$ be an ample line bundle on the projective toric variety $X$.
Is the vector bundle $\cM_L$ semistable, with respect to some choice of polarization?
\end{question}

We mention a similar result for curves: if $X$ is
a smooth projective curve of genus $g$, then $\cM_L$ is a semistable
bundle if either ${\rm deg}(L)\geq 2g+1$ (see \cite{EL}), or if
$L=\omega_X$ (see \cite{PR}).

\bigskip

While we have focused on algebraic notions of positivity, there are parallel notions of positivity in differential geometry, and the relation between these different types of positivity 
is not completely understood.  Recall that a vector bundle is positive in the sense of Griffiths if it admits a Hermitian metric such that the quadratic form associated to the curvature tensor is positive definite.  Griffiths proved that any such positive vector bundle is ample and asked whether the converse is true.  See \cite{Griffiths} for background and further details.  This problem has remained open for forty years and is known in few cases---for very ample vector bundles \cite[Theorem~A]{Griffiths}, and for ample vector bundles on curves \cite{CampanaFlenner}.  Ample toric vector bundles should be an interesting testing ground for the existence of such positive metrics.

\begin{question} \label{q7}
Let $\cE$ be an ample toric vector bundle on a smooth complex projective toric variety.  Is $\cE$ necessarily positive in the sense of Griffiths?
\end{question}

\bigskip

With his differential geometric approach to positivity, Griffiths advanced a program to relate ampleness to numerical positivity.  Fulton and Lazarsfeld completed one part of this program by proving that the set of polynomials in Chern classes that are numerically positive for ample vector bundles are exactly the Schur positive polynomials, the nonzero polynomials that are nonnegative linear combinations of the Schur polynomials \cite{FultonLazarsfeld}.  It is natural to wonder whether this result could have been predicted through a careful study of toric vector bundles.

\begin{question} \label{q8}
Are the Schur positive polynomials the only numerically positive polynomials for ample toric vector bundles?
\end{question}

\noindent It should also be interesting to look for a combinatorial proof that the Schur positive polynomials are numerically positive for ample toric vector bundles.  One natural approach would be to use the characterization of equivariant Chern classes of toric vector bundles in terms of piecewise polynomial functions \cite[Theorem~3]{chow}, together with the combinatorial formulas for localization on toric varieties from \cite{localization}.

\providecommand{\bysame}{\leavevmode \hbox \o3em
{\hrulefill}\thinspace}


\end{document}